\documentclass[12pt]{article}

\usepackage[latin1]{inputenc}
\usepackage[T1]{fontenc}
\usepackage{amsmath,amssymb,amsthm}

\DeclareMathOperator{\coker}{coker}
 \DeclareMathOperator{\HH}{H}
\DeclareMathOperator{\Proj}{Proj} 
\DeclareMathOperator{\Hom}{Hom}
 \DeclareMathOperator{\Ext}{Ext}
\DeclareMathOperator{\codim}{codim}\DeclareMathOperator{\diam}{diam}

\newtheorem{theorem}{Theorem}[section]

\newtheorem{corollary}[theorem]{Corollary}
\newtheorem{definition/corollary}[theorem]{Definition/Corollary}
\newtheorem{definition}[theorem]{Definition}
\newtheorem{example}[theorem]{Example}
\newtheorem{proposition}[theorem]{Proposition}

\newtheorem{remark}[theorem]{Remark}

\usepackage{latexsym}            
\usepackage{amssymb}

\newcommand{\pp}{{\mathbb P}}

\newcommand\sE{{\mathcal E}}
\newcommand\sF{{\mathcal F}}

\newcommand\sI{{\mathcal I}}

\newcommand\sN{{\mathcal N}}
\newcommand\sO{{\mathcal O}}

\newcommand\sS{{\mathcal S}}

\newcommand{\proj}[1]
{ \mathchoice
           { {\mathbb P}^{#1} }
           { {\mathbb P}^{#1} }
           { {\mathbb P}^{#1} }
           { {\mathbb P}^{#1} }
         }

\begin{document}

\title {The Hilbert Scheme of Buchsbaum space curves}
\author{Jan O. Kleppe} 


\date{\ } 

\maketitle 

{\footnotesize \vspace*{-0.40in} \textbf{Abstract.} \noindent We consider the
  Hilbert scheme $\HH(d,g)$ of space curves $C$ with homogeneous ideal
  $I(C):=H_{*}^0(\sI_{C})$ and Rao module $M:=H_{*}^1(\sI_{C})$. By taking
  suitable generizations (deformations to a more general curve) $C'$ of $C$,
  we simplify the minimal free resolution of $I(C)$ by e.g making consecutive
  free summands (ghost-terms) disappear in a free resolution of $I(C')$. Using
  this for Buchsbaum curves of diameter one ($M_v \ne 0$ for only one $v$), we
  establish a one-to-one correspondence between the set $\sS$ of irreducible
  components of $\HH(d,g)$ that contain $(C)$ and a set of minimal 5-tuples
  that specializes in an explicit manner to a 5-tuple of certain graded Betti
  numbers of $C$ related to ghost-terms. Moreover we almost completely (resp.
  completely) determine the graded Betti numbers of all generizations of $C$
  (resp. all generic curves of $\sS$), and we give a specific
  description of the singular locus of the Hilbert scheme of curves of
  diameter at most one. We also prove some semi-continuity results for the
  graded Betti numbers of any space curve under some assumptions.
  \vspace*{-0.12in}
  \thispagestyle{empty} 

  \noindent {\bf AMS Subject Classification.}} {\small 14C05, 14H50, 14M06,
  13D02, 13C40. }

  \vspace*{-0.08in}

\section {Introduction}
The goal of this paper is to give an explicit description of all irreducible
components of the Hilbert scheme $\HH(d,g)$ of space curves
that contain a given Buchsbaum curve. Thus this paper completes the study we
started in \cite{krao} where we only succeeded in some cases (\cite{krao},
Prop. 4.6). Recall that a curve $C$ (equidimensional and locally
Cohen-Macaulay) with sheaf ideal $\sI_C$ is called (arithmetically) Buchsbaum
if the 
Rao module $M:=H_{*}^1(\sI_{C})$ satisfies $(X_0,X_1,X_2,X_3)\cdot M =0$ where
$R = k[X_0,X_1,X_2,X_3]$ is the polynomial ring. Hence if $M_v = 0$ for all
but one $v$, then $C$ is certainly Buchsbaum; we call $C$ a diameter-1 curve
in this case.

In this paper we determine all components $V$ of $\HH(d,g)$ containing a
diameter-1 curve $C$ from the point of view of describing the graded Betti
numbers of the generic curve of $V$ in terms of the graded Betti numbers of
$C$ (Corollary~\ref{cormaxcomp}, Theorem~\ref{Vbetta}). There are 5 graded
Betti numbers of $C$, related to ghost terms if they are non-zero, that play a
very special role and determines for instance the number of components $V$
containing $(C)$ (Proposition~\ref{numberBCM}). Moreover, if $(C)$ is contained
in the closure of a Betti stratum $H(\underline \beta)$, necessarily
irreducible by Proposition~\ref{BettiStrata}, then we determine the set of
graded Betti numbers $\underline \beta$ almost completely
(Theorem~\ref{newmain}, Remark~\ref{remmainthm4}). As a consequence we
describe the singular locus of the Hilbert scheme of curves of diameter at
most one as an explicit union of certain Betti strata, up to closure
(Theorem~\ref{singloc}). To prove such results it is important to understand
which graded Betti number are semi-continuous (Proposition~\ref{newmainprop}).
We also prove a semi-continuity result for the graded Betti numbers for any
space curve under an assumption (Corollary~\ref{semic}). Moreover we need to
find ``all'' generizations of $C$. In \cite{krao} we mainly found the
generizations using some ideas appearing in \cite{MDP1}. In this work we
describe the generization that does not preserve postulation in much more
detail and with a new proof (Proposition~\ref{newmainres}). For the
generization that preserves postulation and reduces $\dim M$ by one, we
correct an inaccuracy in \cite{krao}, Prop. 4.2 (a): the resolution may be
non-minimal in one and only one degree, 
see Remark~\ref{mainres2corr}. All these results, together with those on the
obstructedness and dimension of $\HH(d,g)$ in \cite{krao}, make us understand
the Hilbert scheme of diameter-1 curves.

Thus this paper contributes to solving questions related to the number of
components, irreducibility and smoothness of $\HH(d,g)$, see \cite {A},
\cite{EHM, El}, \cite{Gi, Gu}, \cite{MDP1, MDP2} for some contributions
which are relevant for this paper, and \cite{BM} for a thorough study of
diameter-1 curves.

A part of this work was done during a visit to the Institut Mittag-Leffler
(Djursholm, Sweden) in May 2011, whom I thank for the invitation. I heartily
thank Johannes Kleppe for comments and his contribution, cf.
Proposition~\ref{numberBCM}.
 \vspace*{-0.12in}

\subsection {Notations and terminology} Let $R = k[X_0,X_1,X_2,X_3]$ be a
polynomial ring over an algebraically closed field $k= \overline k$ (of characteristic zero
in the examples) and let $\proj{3}:=\Proj(R)$. A curve $C$ in $\proj{3}$ is an
{\it equidimensional, locally Cohen-Macaulay} (lCM) subscheme of
$\proj{}:=\proj{3}$ of dimension one with sheaf ideal $\sI_C$ and normal sheaf
$\sN_C := \Hom_{\sO_{\proj{}}}(\sI_C,\sO_C)$. If $\sF$ is a coherent
$\sO_{\proj{}}$-Module, we let 
$H_{*}^i(\sF) := \oplus_v H^i(\sF(v))$, $h^i(\sF) := \dim H^i(\sF)$ and
$\chi(\sF) := \Sigma (-1)^i h^i(\sF)$. 
Moreover, $M = M(C)$ is the Hartshorne-Rao module
$H_{*}^1(\sI_{C})$, or just the Rao
module, 
and $I=I(C)$ is the homogeneous ideal $H_{*}^0(\sI_{C})$ of $C$. They are
graded modules over $R$. Note that $M$ is artinian since $C$ is lCM. $C$ is
called ACM (arithmetically CM) if $M=0$. The postulation $\gamma=\gamma_C$
(resp. deficiency $\rho = \rho_C$ and specialization $\sigma = \sigma_C$) of
$C$ is the function defined over the integers by $\gamma(v) = h^0(\sI_C(v))$
(resp. $\rho(v) = h^1(\sI_C(v))$ and $\sigma(v) =
h^{1}(\sO_{C}(v)$)). If $M \ne 0$, let

$ \hspace {2cm} c(C) = \max \{n \arrowvert h^1(\sI_C(n)) \neq 0 \} \ , \quad
b(C)=\min \{n \arrowvert h^1(\sI_C(n)) \neq 0 \} \,, $
\\[3mm]
and let $\diam M := c(C)- b(C)+1 $ be the diameter of $M$ (or of $C$). We say
$C$ has maximal rank if $H^0(\sI_C(c))=0$ where $c =
c(C)$. 
A curve $C$ satisfying $\mathfrak m \cdot M = 0$, $ \mathfrak m =
(X_0,..,X_3)$, is an (arithmetically) Buchsbaum curve, thus {\it diameter-1
  curves} are necessarily
Buchsbaum. 
 
 We say $C$ is {\it unobstructed\/} if the Hilbert scheme (\cite{G}) of space
 curves of degree $d$ and arithmetic genus $g$, $\HH(d,g)$, is smooth at the
 corresponding point $(C)$, otherwise $C$ is {\it obstructed}. The open part
 of $\HH(d,g)$ of {\it smooth connected} space curves is denoted by
 $\HH(d,g)_S$, while $\HH_{\gamma,\rho} = \HH(d,g)_{\gamma,\rho}$ (resp.
 $\HH_{\gamma}$) 
 denotes the subscheme of $\HH(d,g)$ of curves with constant
cohomology, i.e. $\gamma_C$ and $\rho_C$ do not vary with $C$ (resp. constant
postulation $\gamma$), 
cf. \cite{MDP1} for an introduction. Let $V$ be an irreducible subset (resp.
component) of $\HH(d,g)$ containing $(C)$. A curve in a sufficiently small
open subset $U$ of $V$ (small enough so that any curve in $U$ has all the
openness properties that we want to require) is called a {\it generization}
of $C \subseteq \proj{3}$ in $\HH(d,g)$ (resp. a {\it generic} curve of
$\HH(d,g)$). We define generizations in $\HH_{\gamma}$ and $\HH_{\gamma,\rho}$
similarly.
%
 \vspace*{-0.12in}

\section {Background}
 \vspace*{-0.10in}
In this section we review techniques and results which we will
need in this paper. 
 \vspace*{-0.08in}

\subsection {Minimal resolutions and graded Betti numbers}
Let $C$ be a curve in $\proj{3}$. Then the homogeneous ideal $I=I(C)$ has a 
minimal resolution of the following form
\begin{equation} \label{resolu} 0 \rightarrow \oplus_i
  R(-i)^{\beta_{3,i}} \rightarrow \oplus_i R(-i)^{\beta_{2,i}} \rightarrow
  \oplus_i R(-i)^{\beta_{1,i}} \rightarrow I \rightarrow 0 \, .
\end{equation}
The numbers $\beta_{j,i}=\beta_{j,i}(C)$ are the graded Betti numbers of
$I(C)$. We denote the set of all graded Betti numbers of $I(C)$ by $\underline
\beta(C) := \{ \beta_{j,i}(C)\}$. We define the {\it Betti stratum},
$\HH(\underline \beta)$, of
$\HH(d,g)_{\gamma,\rho}$ to consist of all curves $C$ of
$\HH(d,g)_{\gamma,\rho}$ 
satisfying $\beta_{j,i}(C) = \beta_{j,i}$ for every $i,j$.

Now we recall Rao's theorem concerning the form of a minimal resolution
of $I=I(C)$. Let \begin{equation}\label{resoluM} 0 \rightarrow
  L_4 \xrightarrow{\ \sigma} L_3 \rightarrow L_2 \rightarrow L_1
   \xrightarrow{\ \tau} L_0 \rightarrow M \rightarrow 0
\end{equation}
be the minimal resolution of $M = M(C)= H_{*}^1(\sI_C)$ and let $L_j=
\oplus_i R(-i)^{\beta_{j+1,i}(M)}$. Then \eqref{resolu} and
\begin{equation} \label{resoluMI}
  0 \rightarrow L_4 \xrightarrow{\ \sigma \oplus 0} L_3 \oplus F_2
    \longrightarrow F_1 \rightarrow I \rightarrow 0  \ 
\end{equation}
are isomorphic (\cite{R}, Thm. 2.5)! Here the composition of $L_4 \rightarrow
L_3 \oplus F_2$ with the natural projection $L_3 \oplus F_2 \rightarrow F_2 $
is zero. We may write \eqref{resoluMI} as a so-called {\em $E$-resolution} of
$I$ (cf. \cite{MDP1}):
\begin{equation} \label{resoluE} 0 \rightarrow E \oplus F_2 \rightarrow F_1
  \rightarrow I \rightarrow 0 \ , \qquad E:= \coker \sigma \, .
\end{equation}
For a diameter-1 curve $C$ with $r=\dim
H_*^1(\sI_C)=h^1(\sI_C(c))$, we have the free resolution {\small
   \begin{equation} \label{Koszul} 0 \rightarrow R(-c-4)^r
     \xrightarrow{\sigma} R(-c-3)^{4r} \rightarrow R(-c-2)^{6r} \rightarrow
     R(-c-1)^{4r} \rightarrow R(-c)^r \rightarrow M \rightarrow
     0 \end{equation}}which is ``$r$ times'' the Koszul resolution of the
 $R$-module $k \cong R/\mathfrak m$ twisted by $-c$. Hence we may put
 $\oplus_i R(-i)^{\beta_{3,i}}= R(-c-4)^r$ in \eqref{resolu}. If $r=1$ then
 the matrix of $\sigma$ is just the transpose of $(X_0, X_1, X_2, X_3)$. 
 \begin{example} \label{BKM1} There is a 
   curve in $\HH(33,117)_S$
of $\diam M = 1$ with minimal resolution
\begin{equation*}
  0 \rightarrow R(-9) \rightarrow R(-10)^{2} \oplus  R(-9)
  \oplus  R(-8)^{4} \rightarrow 
  R(-9) \oplus  R(-8) \oplus
  R(-7)^{5} \rightarrow I \rightarrow 0 \ , 
  \end{equation*}
  (see \cite{BKM} or \cite{W2}). If we compare it to the Rao form
  \eqref{resoluMI}, we see that $F_2= R(-10)^{2} \oplus R(-9)$ and that $ 0
  \rightarrow L_4=R(-9) \rightarrow L_3 = R(-8)^{4}$ is the leftmost part in
  the minimal resolution of $M$. Note that $F_2$ and $L_4$ have the common
  free summand $R(-9)$. A repeated summand in two consecutive terms in
  the minimal resolution \eqref{resolu} will be called a {\rm ghost term}. 
Also $F_1$ and $F_2$ have $R(-9)$ as a ghost
  term. 
 \end{example}  

 \begin {definition} \ The Rao module $M=M(C)$ admits ``a Buchsbaum
   component'' $ M_{[t]}$ if  \\[2mm]
   \centerline {$M \simeq M' \oplus M_{[t]}$} \\[2mm] as graded $R$-{\it
     modules} where $ M_{[t]}$ is the graded R-module $k$ supported in degree
   $t$ ($ M_{[t]} \cong k(-t)$). 
\end{definition}

\begin{remark} Suppose $M=M(C)$ admits a Buchsbaum component, $M \simeq M'
  \oplus M_{[t]}$.

  {\rm (a)} If $ M'$ is a direct sum of other Buchsbaum components of possibly
  various degrees (resp. of the same degree $t$, i.e. $M \simeq
  M_{[t]}^r$), then $C$ is a Buchsbaum curve (resp. of diameter one).

  {\rm (b)} Buchsbaum curves are only a special class of curves having
  Buchsbaum components. Indeed every curve obtained from Liaison addition
  where one of the curves is Buchsbaum, has a Buchsbaum component up to a
  possible twist (see \cite{MIG} for the notion of Liaison addition).
\end{remark}

If $M \simeq M' \oplus M_{[t]}$ and if we denote $(\sigma',\sigma_{[t]}) :=
\left(\begin{smallmatrix} \sigma' & 0 \\ 0 &
    \sigma_{[t]} \end{smallmatrix}\right)$,
then $M$ has the minimal resolution: {\small
  \begin{equation} \label{MplussKosz} 0 \rightarrow P_4 \oplus
    R(-t-4)\xrightarrow{(\sigma',\sigma_{[t]})} P_3 \oplus R(-t-3)^{4}
    \rightarrow P_2 \oplus R(-t-2)^{6} \rightarrow ... \rightarrow P_0 \oplus
    R(-t) \rightarrow M \rightarrow 0 \end{equation} } where $ 0 \rightarrow
P_4 \xrightarrow{\sigma'} P_3 \rightarrow P_2 \xrightarrow{\tau_2} P_1
\xrightarrow{\tau_1} P_0 \rightarrow M' \rightarrow 0$ is a minimal resolution
of $M'$ and {\small
  \begin{equation} \label{Kosz} 0 \rightarrow R(-t-4)
    \xrightarrow{\sigma_{[t]}} R(-t-3)^{4} \rightarrow R(-t-2)^{6} \rightarrow
    R(-t-1)^{4} \xrightarrow{\tau_{[t]}} R(-t) \rightarrow M_{[t]} \rightarrow
    0 \end{equation} } is the Koszul resolution of the $R$-module $R/\mathfrak
m(-t)$. Note that $\sigma_{[t]}=(X_0, X_1, X_2, X_3)^{tr}=\tau_{[t]}^{tr}$.
Combining with Rao's theorem
concerning \eqref{resoluMI}, we get the following minimal resolution of $I$:
\begin{equation}  \label{resoluMt}
  0 \rightarrow P_4 \oplus R(-t-4)  \xrightarrow{(\sigma',\sigma_{[t]}) \oplus
    0} P_3
  \oplus R(-t-3)^{4} \oplus F_2
  \rightarrow F_1 \rightarrow I \rightarrow 0 \, . \ 
  \end{equation}
  It was shown in \cite{krao} that certain Betti number were related to whether
  $(C)$ sits in the intersection of different irreducible components of
  $\HH(d,g)$, and hence to whether $C$ is obstructed, or not. To define them,
  we write $F_i$ as
 \begin{equation} \label{bettiMI} F_2 \cong Q_2 \oplus R(-t-4)^{b _1} \oplus
   R(-t)^{b _2} \ , \qquad F_1 \cong Q_1 \oplus R(-t-4)^{a_1} \oplus
   R(-t)^{a_2}
  \end{equation}
  where $Q_i$, for $i = 1,2$ are supposed to contain no free direct summand of
  degree $t$ and $t+4$.
\begin{definition} The 4-tuple associated to a curve $C$ with 
  Buchsbaum component $ M_{[t]}$ is $(a_1,a_2,b_1,b_2)$.
  Note that $(a_1,a_2)=(\beta_{1,t+4},\beta_{1,t})$ are the $1^{st}$ graded
  Betti numbers of $I=I(C)$. 
\end{definition} 
\begin{remark} \ For a Buchsbaum curve of diameter one, we have $M(C) \simeq
  M_{[t]}^r$ and $t=c$. Then $(a_1,a_2,b_1,b_2) =(\beta_{1,c+4},\beta_{1,c},
  \beta_{2,c+4}, \beta_{2,c})$ and $r=\beta_{3,c+4}$ are the graded Betti
  numbers of $I(C)$ in degree $c+4$ and $c$. In this case, if we want to have
  $r$ attached, we work with the 5-tuple $(a_1,a_2,b_1,b_2, r)$. Note that
  this 5-tuple was denoted by $(r,a_1,a_2,b_1,b_2)$ in \cite{krao}.
\end{remark}
\subsection {Linkage}
We will need the notion of linkage and how we can find the minimal resolution
of a linked curve (cf. \cite{PS} and see \cite{MIG} for an introduction to
linkage or liaison). 
Considering $ {\mathcal
  I}_{C/Y}:= {\mathcal I}_{C}/{\mathcal I}_Y$ as the sheaf ideal of $C$ in
$Y$, we define
\begin{definition}\label{deflink}  Two curves $C$ and $D$  in $\proj{3}$ are
  said to be (algebraically) {\it CI-linked} if there exists a complete
  intersection curve (a CI) $Y$ such that $${\mathcal I}_{C}/{\mathcal I}_Y
  \cong {\mathcal H}om_{{\mathcal O}_{\pp}}({\mathcal O}_{D},{\mathcal O}_Y)\
  \ {\rm \ and \ \ } {\mathcal I}_{D}/{\mathcal I}_Y \cong {\mathcal
    H}om_{{\mathcal O}_{\pp}}({\mathcal O}_{C},{\mathcal O}_Y) \, . $$
\end{definition}
Suppose that $Y$ is a complete intersection of two surfaces of degrees $f$ and
$g$ (a CI of type $(f,g)$) containing $C$. Since the dualizing sheaf,
$\omega_Y$, of $Y$ satisfies $\omega_Y \cong \sO_Y(f+g-4)$, we get
\begin{equation} \label{linkdual}
{\mathcal I}_{C/Y} \cong \omega_{D}(4-f-g) \qquad {\rm and}
  \qquad {\mathcal I}_{D/Y} \cong \omega_{C}(4-f-g) 
\end{equation}
%
from the definition. By \cite{R} the module $M(C)$ is a biliaison (linking
twice several times) invariant, up to twist. Moreover, using \eqref{linkdual}
and the fact that $ \omega_{D} \cong {\sE}xt^2(\sO_{D}, \sO_{\pp}(-4))$, hence
that $I(C)/I(Y) \cong \Ext^1(I_D(f+g), R)$,
one knows how to find a resolution of $I(D)$ in terms of the
resolution of $I(C)$ and some part of the resolution of the dual of $M(C)$. 
Indeed using the $E$-resolution of $I(C)$, there exists vertical
morphisms
\begin{gather*} \label{bigdia}
  0 \longrightarrow R(-f-g) \rightarrow R(-f) \oplus R(-g) \rightarrow
  I(Y)   \longrightarrow 0 \\
   \hspace{16pt} \downarrow \hspace{32pt} \circ \hspace{30pt} \downarrow
  \hspace{30pt} \circ \hspace{22pt} \downarrow \\
   0 \longrightarrow \quad E \oplus F_2   \hspace{18pt}  \longrightarrow
   \hspace{18pt} F_1  \hspace{15pt} \longrightarrow  \hspace{15pt} I(C)
   \longrightarrow 0 
\end{gather*}
The mapping cone construction yields a resolution of $I(C)/I(Y)$. Taking
$R$-duals, $ {\Hom_R}(-,R)$, of it, we get
\begin{equation} \label{resoluLink} 0 \rightarrow F_1^{\vee} \rightarrow
  E^{\vee} \oplus F_2^{\vee} \oplus R(f) \oplus R(g)  \to I(D)(f+g)
  \to 0 \, .
\end{equation} 
Note that $0 \rightarrow L_0^{\vee} \xrightarrow{\tau^{\vee}} L_1^{\vee}
\rightarrow L_2^{\vee} \rightarrow E^{\vee} \to 0$ is a free resolution of $
E^{\vee}$ because the $R$-dual sequence of \eqref{resoluM} is a resolution of
$\Ext_R^4(M , R)$. Letting $G_1:= L_2^{\vee}(-f-g) \oplus F_2^{\vee}(-f-g)
\oplus R(-g) \oplus R(-f)$, the mapping cone construction yields the following
$R$-free resolution:
\begin{equation} \label{resoluFreeLink} 0 \rightarrow L_0^{\vee}(-f-g)
  \xrightarrow{\tau^{\vee} \oplus 0} L_1^{\vee}(-f-g) \oplus F_1^{\vee}(-f-g)
  \rightarrow G_1 \to I(D) \rightarrow 0 \
\end{equation}
%
If we need to find a free resolution of the homogeneous ideal of  a curve $X$
linked to $D$, using a CI $Z$ of type $(f',g')$ (so $X$ and $C$ are bilinked),
we use \eqref{resoluLink} (and not \eqref{resoluFreeLink}) and the mapping
cone construction as in the big diagram above, 
to find a resolution of $I(D)/I(Z)(f'+g')$. Taking $R$-duals we get a free
resolution of $R/I(X)$ (cf. \cite{MIG}). We illustrate this by an example:
\begin{example} \label{ex1} 
  If $C$ is a disjoint union of two lines, then it is easy to see that
 \begin{equation*} 0 \rightarrow R(-4) \xrightarrow{\sigma} R(-3)^{4}
    \rightarrow R(-2)^4 \rightarrow I(C) \rightarrow 0 
    \end{equation*} 
    is the minimal resolution, having \ $0 \to E \rightarrow R(-2)^4
    \rightarrow I(C) \rightarrow 0 $ as its $E$-resolution (cf.
    \eqref{resoluE}). We link twice, first via a CI of type $(4,2)$ to get a
    curve $D$ with an exact sequence (cf. \eqref{resoluLink}) $$0 \rightarrow
    R(2)^3 \rightarrow E^{\vee} \oplus R(4) \to I(D)(6) \to 0 \, ,$$ then we
    link via a CI \ $Z$ of type $(4,6)$ to get a curve $X$ in $\HH(18,39)$ with
    $E$-resolution: $$0 \to E(-4) \oplus R(-8) \rightarrow R(-6)^4 \oplus
    R(-4) \rightarrow I(X) \rightarrow 0\, ,$$ which really is the $R$-dual
    sequence of the resolution of $I(D)/I(Z)(10)$ found by the mapping cone
    construction. Note that we use a common hypersurface of degree $4$ in both
    linkages. The minimal resolution of $I(X)$ is
    \begin{equation*} 0 \rightarrow R(-8) \rightarrow R(-8) \oplus R(-7)^{4}
      \rightarrow R(-6)^4 \oplus R(-4) \rightarrow I(X) \rightarrow 0 \, .
    \end{equation*} 
    One should compare the resolution with the Rao form
    \eqref{resoluMI}. Note that $R(-8)$ is a ghost term.
\end{example}

\subsection {Deformations}
In \cite{krao} we proved that we can cancel repeated free consecutive
summands (ghost terms) in \eqref{resoluMI} using deformations:

\begin{theorem}\label{mainres} \quad
  Let $C \subseteq \proj{3}$ be any curve with homogeneous ideal $I(C)$ and
  Rao module $M(C)$ and minimal free resolutions as in \eqref{resoluM} and
  \eqref{resoluMI}. If $F_1$ and $F_2$ have a common free summand; 
  $F_2 = F_2' \oplus R(-i)$, $F_1 = F_1' \oplus R(-i)$, 
  then there is a generization $C'$ of $C$ in $\HH(d,g)$ 
  with constant postulation and constant Rao module, and with minimal
  resolution
\begin{equation*} 
  0 \rightarrow L_4  \xrightarrow{ \ \sigma \oplus 0 \ } L_3 \oplus F_2'
    \rightarrow F_1' \rightarrow I(C') \rightarrow 0 \, . \ 
\end{equation*}
\end{theorem}

The proof is straightforward once we have proven a key lemma, and we refer to
\cite{krao}, Thm. 4.1 for the details. We remark that the proof of the case $M
\cong k(-c)$ in \cite{MDP1} extends to get Theorem~\ref{mainres}.

\begin{corollary} \label{mainrecor} \ Let $C$ be any curve and let $ \{
  \beta_{j,i}\}$ (resp. $ \{ \beta_{j,i}(M)\}$) be the graded Betti numbers of
  $I(C)$ (resp. $M(C)$, whence  $L_3= \oplus_i R(-i)^{\beta_{4,i}(M)}$). If
  $\beta_{1,i} \cdot (\beta_{2,i}-\beta_{4,i}(M)) \neq 
  0$ for some $i$, then there is a generization $C'$ of $C$ in $\HH(d,g)$ with
  constant postulation and Rao module whose graded Betti
  numbers  $ \{ \beta'_{j,i}\}$ satisfy: \\[-2mm]

  \qquad \qquad \,$\beta'_{1,i} = \beta_{1,i}-1$\,, \quad $\beta'_{1,j} =
  \beta_{1,j}$ \ for \ $j \ne i$

  {\rm (Qi)} \qquad $\beta'_{2,i} = \beta_{2,i}-1$\,, \quad $\beta'_{2,j} =
  \beta_{2,j}$ \ for \ $j \ne i$\ , \quad and \quad \,$\beta'_{3,j} =
  \beta_{3,j}$ \ for \ every $j$. \\[2mm]
  In particular if $C$ is a generic curve of $\HH(d,g)$, then $\beta_{1,i}
  \cdot (\beta_{2,i}-\beta_{4,i}(M))=0$ for every $i$.
\end{corollary}

\begin{proof} By the  semi-continuity of the postulation, a generic curve
  belongs to some open irreducible subset $U$ of $\HH(d,g)$ with {\it
    constant postulation}. It follows that $\beta_{1,i}$ is
  semi-continuous in $U$, cf.\;the proof of Corollary~\ref{semic} for a
  discussion. Hence also the final statement of the corollary is immediate.
\end{proof}
In \cite{krao}, Prop. 4.2 (a) we also proved the following result.

\begin{proposition} \label{mainres1} Let $C \subseteq \proj{3}$ be a curve for
  which there is an isomorphism $M(C) \cong M' \oplus M_{[t]}$ as graded
  $R$-{\it modules} such that the minimal resolution \eqref{resoluMt} of
$I(C)$ takes the form:
  \begin{equation} \label{resoluMI4} 0 \rightarrow P_4 \oplus R(-t-4)
    \xrightarrow{(\sigma', \sigma_{[t]}) \oplus 0 \oplus 0} P_3 \oplus
    R(-t-3)^{4} \oplus Q_2 \oplus R(-t-4)   \xrightarrow{\ \beta} F_1
    \rightarrow I(C) 
    \rightarrow 0 \, . \
\end{equation}
Then there is a generization $C' \subseteq \proj{3}$ of $C \subseteq \proj{3}$
in $\HH(d,g)$ with constant postulation such that $I(C')$ has a free
resolution of the following form:
\begin{equation} \label{resoluMI55}
  0 \rightarrow P_4 
  \xrightarrow{\sigma' \oplus 0 \oplus 0} P_3 \oplus R(-t-3)^{4}
  \oplus  Q_2  \rightarrow F_1 \rightarrow I(C') \rightarrow 0 \, , \ 
\end{equation}
and such that $M(C') \cong M'$ as graded $R$-modules. The
resolution is minimal except possibly in degree $t+3$ in which case some of
the summands of $R(-t-3)^{4}$ may be cancelled against free summands of
$F_1$.
\end{proposition}
\begin{proof}[Proof (the main step)] 
  We replace the $0$-coordinate in the matrix of $(\sigma', \sigma_{[t]})
  \oplus 0 \oplus 0$ that corresponds to $ R(-t-4) \rightarrow R(-t-4)$, by
  some indeterminate $\lambda$ of degree zero (as in \cite{MDP1}, page 189).
  To get a complex in \eqref{resoluMI4}, we change the four columns
  $\{h_{j,0},h_{j,1},h_{j,2},h_{j,3}\}$ in the matrix of $\beta$,
  corresponding to the map $R(-t-3)^4 \rightarrow F_1$, as follows. Look at
  the column $\{y_j\}$ of the map $R(-t-4) \rightarrow F_1$ induced by
  $\beta$, and put $y_j= \sum_{i=0}^3 a_{j,i} X_i$ (such $a_{j,i} \in k$
  exist, and exactly here we use that the resolution is minimal because we
  need $y_j=0$ when $y_j \in k$). If we replace the four columns above by
  $\{h_{j,0}-\lambda \cdot a_{j,0}, h_{j,1}- \lambda \cdot a_{j,1},
  h_{j,2}-\lambda \cdot a_{j,2}, h_{j,3}-\lambda \cdot a_{j,3}\}$, leaving the
  rest of $\beta$ unchanged, we get that the changed sequence
  \eqref{resoluMI4} defines a complex, and we conclude by e.g. \cite{krao},
  Lem. 4.8.
\end{proof}

\begin{remark} \label{mainres2corr} In \cite{krao}, Prop. 4.2 (a) the
  resolution \eqref{resoluMI55} was claimed to be minimal. The proof of
  \cite{krao}, Prop. 4.2 (a) only supports the minimality in degrees $\neq
  t+3$, leaving the possibility of some of the summands of $R(-t-3)^{4}$ to be
  cancelled against corresponding summands of $F_1$. This explains why we in
  Proposition~\ref{mainres1} have to correct the conclusion of \cite{krao},
  Prop. 4.2 (a). Only the mentioned result of \cite{krao} needs a correction.
  Moreover if $F_1$ does not contain $R$-free summands of the form $R(-t-4)$,
  then it is not necessary to assume that \eqref{resoluMI4} is minimal (cf.
  the proof above and note that one may show \cite{krao}, Lem. 4.8 for
  possibly non-minimal resolutions). The final sentence of
  Proposition~\ref{mainres1} requires, however, that \eqref{resoluMI4} is
  minimal.
\end{remark}

\begin{corollary} \label{mainres12} Let $M(C) \cong M' \oplus M_{[t]}$ as
  graded $R$-modules and let $(a_1,a_2,b_1,b_2)$ be the corresponding 4-tuple.
  If $b_1 \neq 0$, then there is a generization $C'$ of $C$ in $\HH(d,g)$ with
  constant postulation and $M'$ whose 4-tuple is \\[2mm] {\rm (P1)}
  \hspace{4cm} $(a_1,a_2,b_1-1,b_2)\, .$ \\[2mm] Indeed for $i \in \{1,2\}$,
  $h^i(\sI_{C'}(v)) = h^i(\sI_{C}(v))$ for $v \neq t$ and $h^i(\sI_{C'}(t)) =
  h^i(\sI_{C}(t))-1$.
\end{corollary}

In \cite{krao}, Cor.\;3.3 and Thm.\;3.4 we saw that the $4$-tuple was important
for discovering obstructedness:
\begin{corollary} \label{introth3} Let $C$ be a curve for which there is a
  graded $R$-module isomorphism $M(C) \cong M' \oplus M_{[t]}$, let
  $(a_1,a_2,b_1,b_2)$ be the 4-tuple and suppose ${_0\!\Ext_R^2}(M,M)=0$. Then
  $C$ is obstructed if
  $$
  a_2 \cdot b_1 \neq 0 \quad {\rm or} \quad a_1 \cdot b_1 \neq 0 \quad {\rm
    or} \quad a_2 \cdot b_2 \neq 0 \, .$$ 
  Moreover if $C$ is a diameter-1 curve (whence $t=c$), then $C$ is
  obstructed if and only if \small {$$ \beta_{1,c} \cdot \beta_{2,c+4} \neq 0
    \ \ \ {\rm or} \ \ \ \beta_{1,c+4} \cdot \beta_{2,c+4} \neq 0 \ \ \ {\rm
      or} \ \ \ \beta_{1,c} \cdot \beta_{2,c}\neq 0 \, . $$ }
\end{corollary}

\begin{remark} \label{rem5tupleobstr} Let $M(C) \cong M' \oplus M_{[t]}$ as
  graded $R$-modules. For its 4-tuple $(a_1,a_2,b_1,b_2)$, we have that $$ a_2
  \cdot b_1 = 0 \quad {\rm and} \quad a_1 \cdot b_1 = 0 \quad {\rm and} \quad
  a_2 \cdot b_2= 0 \
$$ is equivalent to requiring it to be of the form $
(0,0,b_1,b_2) , (a_1,0,0,b_2) {\rm \ or} \ (a_1,a_2,0,0).$ Hence by
Corollary~\ref{introth3}, if $C$ is unobstructed, then there are ``two
consecutive 0's in the 4-tuple''. This is equivalent to unobstructedness if
$\diam M = 1$. Note that if $\diam M = 0$ ($C$ is ACM), then $C$ is always
unobstructed by \cite{El}.
\end{remark}

\begin{example} \label{ex3} {\rm (a)} Start with the generic curve of
  $\HH(8,5)_S$. It has 2-dimensional Rao module $M$ and $\diam M =1$ by 
  \cite{GP1}. We link with a CI of type $(4,6)$, then with a CI of type
  $(6,8)$, using the same degree-6 surface in both linkages. The minimal
  resolution of the bilinked curve is
   \begin{equation*} 
      0 \rightarrow R(-10)^2 \rightarrow R(-10) \oplus  R(-9)^{8} \rightarrow
      R(-8)^7 \oplus  R(-6) \rightarrow I \rightarrow 0  \, ,
    \end{equation*} 
    whence $c=6$ and $r=2$. The corresponding 4-tuple is $ \
    (\beta_{1,c+4},\beta_{1,c},\beta_{2,c+4}, \beta_{2,c}) = (0,1,1,0),$ i.e.
    the curve $C$ of \ $\HH(32,109)_S$ is obstructed by
    Remark~\ref{rem5tupleobstr}.

    {\rm (b)} The curve $C$ of \ $\HH(33,117)_S$ of \ $\diam M = 1$ of
    Example~\ref{BKM1} has 4-tuple $(1,0,1,0)$, i.e. $C$ is obstructed by
    Remark~\ref{rem5tupleobstr}. Since $c(C)=5$, this curve has
    maximal rank.
%
\end{example}

In the next section, we shall see that the curve of Example~\ref{ex3} (a)
belongs to a unique irreducible component, while the curve of
Example~\ref{ex3} (b) sits in the intersection of two irreducible components
of  $\HH(d,g)_S$.

\section{On the semi-continuity of graded Betti numbers} 
The goal of this section is to show a result on the semi-continuity of the
graded Betti numbers of the homogeneous ideal $I(C)$ of a curve $C \subseteq
\pp^3$ considered as a point in $\HH(d,g)$. We get the result as a consequence
of the fact that the immersion $ \HH_{ \gamma} \rightarrow \HH(d,g)$ is an
isomorphism in an open neighbourhood of $(C)$ under a certain assumption. We
also show a variation of a result of Bolondi, leading to the irreducibility of
Betti strata with constant Rao modules. Letting ${_0\!\Ext_{R}^i}(-,-)$ be the
degree-$0$ part of $\Ext_{R}^i(-,-)$, we have

\begin{theorem} \label{gen}
  Let $C$ be any curve and let $I = H_{*}^0(\sI_C)$  and  $M =
  H_{*}^1(\sI_C)$. Then \\[-1mm] 

   \qquad ${_0\!\Hom_R}(I,M) = 0 \ \ \Longrightarrow \ \ \HH_{ \gamma}
   \cong
   \HH(d,g) \ \ {\rm are \ isomorphic \ as \ schemes \ at} \ (C) . $ 
\end{theorem}
\begin{proof} By mainly interpreting the exact sequence
\begin{equation} \label{mseq1}
0 \to {_0\!\Ext_R^1 }(I,I ) \rightarrow  H^0({\sN_C}) \rightarrow
 {_0\!\Hom_R}(I , M)  \rightarrow {_0\!\Ext_R^2 }(I,I )
  \rightarrow H^1({\sN_C}) \rightarrow
\end{equation} in terms of deformation theories, as done in Prop.\;2.10 of
\cite{krao}, we get the conclusion.
\end{proof}

%
\begin{remark} \label{genPis} Theorem~\ref{gen} holds in general for any
  closed subschemes $C$ of $\mathbb{P}^n_k=\Proj(R)$, $k= \overline k$ under
  the sole assumption ${_0\!\Hom_R}(I,M) = 0$ by \cite{K07}, Prop.\;8 where
  the main ingredient in a proof (the isomorphism between the local graded
  deformation functor of $R \to R/I$ and the local Hilbert functor of $C
  \subset \mathbb{P}^n$) was proven already in 1979 (\cite{K79}, Thm.\;3.6 and
  Rem.\;3.7). Note that if $H^1(\sI_C(\deg F_i))= 0$ for every minimal
  generator $F_i$ of $I$, we get ${_0\!\Hom_R}(I,M)=0$ and
  hence 
  this result generalizes the comparison theorem of Piene-Schlessinger in
  \cite{PiS}.
\end{remark}

If $C$ has maximal rank, then $ {_0\!\Hom_R}(I , M) = 0$. In this case it is
not so difficult to show $ \HH_{ \gamma} \cong \HH(d,g) $ at $(C)$ by using
the semi-continuity of $h^i(\sI_C(v))$. The assumption $ {_0\!\Hom_R}(I , M) =
0$ are, however, much weaker than requiring $C$ to be of maximal rank, at
least for generic unobstructed curves. In fact if ${_0\!\Ext_R^2}(M , M) = 0$
and $C$ is unobstructed and generic in $\HH(d,g)$, then it is shown in
\cite{krao}, Prop.\;2.11 that $ {_0\!\Hom_R}(I , M) = 0$.

As a surprising consequence of Theorem~\ref{gen}, we get the following result
on the semi-continuity of the graded Betti numbers which we heavily use in the
next section.

\begin{corollary} \label{semic} Inside $\HH_{ \gamma}$ and hence inside $\HH_{
    \gamma, \rho}$ the graded Betti numbers are upper semi-continuous, i.e. if
  $C'$ is a generization of $C$ in $\HH_{ \gamma}$, then $$\beta_{i,j}(C') \le
  \beta_{i,j}(C) \quad {\rm for \ any \ } i,j \ .$$ In particular if $C$ is
  any curve satisfying $ {_0\!\Hom_R}(I(C),M(C)) = 0$, then $\beta_{i,j}(C')
  \le \beta_{i,j}(C)$ for any i,j and every generization $C'$ of $C$ in
  $\HH(d,g)$.
\end{corollary}

\begin{proof} We apply Nakayama's lemma to the syzygy modules of
  \eqref{resolu} as explained in \cite{K07}, Rem.\! 7 where we to a certain
  degree use \cite{RZ}, but our use of semi-continuity which takes place in a
  flat family with constant postulation is well known \cite{BG}. Then
  we combine with Theorem~\ref{gen}.
\end{proof}

\begin{example} \label{ex2} It is known that the curve $X$ of
  Example~\ref{ex1} sits in the intersection of two irreducible components of
  $\HH(18,39)_S$ and that the generic curve $\tilde X$ of one of the
  components satisfies 
  \begin{equation*} 0 \rightarrow R(-8) \oplus R(-6)^{2}
    \rightarrow R(-5)^4 \rightarrow I(\tilde X) \rightarrow 0 \, . 
    \end{equation*} 
    (Sernesi \cite{Se}, cf. \cite{EF}). Looking at the minimal resolution of
    $I(X)$ in Example~\ref{ex1}, we get $ \beta_{1,5}( X)= \beta_{2,6}( X)= 0$
    while $\beta_{1,5}(\tilde X)= 4$, $\beta_{2,6}(\tilde X)= 2$ , i.e. we
    don't have semi-continuity for $\beta_{1,5}$ and $\beta_{2,6}$. In this
    example Corollary~\ref{semic} does not apply because
    ${_0\!\Hom_R}(I(X),M(X))\ne 0$!
\end{example}

Finally we consider the {\it Betti stratum} $ \HH(\underline \beta):= \left\{
  (D) \in \HH_{\gamma,\rho} \ \arrowvert \ \beta_{j,i}(D) = \beta_{j,i} {\rm \
    for \ every} \ i,j \right\}$, see \cite{I} and its references for papers
on the Betti stratum. Thanks to Bolondi's proof of the irreducibility of
$\HH_{ \gamma, \rho}$ in the Buchsbaum case (\cite{B}, Thm.\;2.2, cf.
\cite{BM2}, Prop.\;4.3), we easily get
\begin{proposition} \label{BettiStrata} If\, $C \subseteq \pp^3$ is a diameter-1
  curve or\, $C$ is ACM, then $ \HH(\underline \beta(C))$ is irreducible.
\end{proposition}

\begin{proof} Suppose $ \HH(\underline \beta(C))$ is not irreducible,
  containing at least two different irreducible components with generic curves
  $D_1$ and $D_2$. Then $D_1$ and $D_2$ have exactly the same $R$-free
  summands and the same morphism $\sigma \oplus 0$ in the minimal resolution
  \eqref{resoluMI}, cf.\;\eqref{Koszul}, but the maps $L_3 \oplus F_2 \to F_1$
  are different. In their $E$-resolutions the curves correspond to two maps
  $\varphi_{D_1}$ and $\varphi_{D_2}$ in $ \Hom(E \oplus F_2,F_1)$, $E=\coker
  \sigma$. Consider the deformation induced by
  \begin{equation} \label{generiT}
    \varphi_t:=t\varphi_{D_1}+(1-t)\varphi_{D_2} \in \Hom(E \oplus F_2,F_1), \
    t \in \mathbb A^1_k.
    \end{equation} 
  In some open subset $U \subset \mathbb A^1_k$
  containing $0$ and $1$, $\varphi_t$ defines a curve with the same graded
  Betti numbers as $D_1$ (and $D_2$) because in the minimal resolutions where
  $0$-entries occur for the matrices of $\varphi_{D_1}$ and $\varphi_{D_2}$
  due to repeated direct summands of $F_2$ and $F_1$, the same entry also
  vanishes for $\varphi_t$. Since $U$ is irreducible we are
  done. 
\end{proof}
\begin{definition} \label{genT} If $(D_1), (D_2) \in \HH(\underline \beta)$
  are related as in \eqref{generiT}, then the generic element $\tilde D$ of $
  \mathbb A^1_k$ is called a trivial generization of $D_1$ (or of $D_2$).
  Obviously, $(\tilde D) \in \HH(\underline \beta)$.
\end{definition}
\begin{corollary} [of proof] \label{gencorT} Two arbitrary curves $D_1$ and
  $D_2$ of $\HH(\underline \beta)$ 
  admit a trivial generization.
\end{corollary}
\begin{remark} [Bolondi, cf. \cite{B}, Cor.\;2.3] \label{BettiStrata2}
  Let $C \subseteq \pp^3$ be any curve with Rao module $M$. By the same proof
  as above we get the irreducibility
  of: 
  \[ \left\{ (D) \in \HH_{\gamma} \ \arrowvert \ M(D) \simeq M \ {\rm as \
      graded \ }R{\rm-modules, \ and} \ \beta_{j,i}(D) = \beta_{j,i}(C) {\rm
      \ for \ every} \ i,j \right\} \, .
\]
\end{remark}

\section{Generizations not preserving postulation}
In this section we study generizations of space curves, i.e.\;deformations to
more general curves by ``simplifying'' their minimal resolutions. We start
with the following generalization of \cite{krao}, Prop.\;4.2 (b) for which we
give a new proof 
where we make ghost terms of a linked curve redundant under generization. Note
that by {\em redundant} terms in a free resolution, we mean consecutive free
summands that split off (disappear) when we make the free resolution minimal,
while ghost terms don't split off! Recalling $ M_{[t]} \cong R/\mathfrak
m(-t)$, 
we have

\begin{proposition} \label{newmainres} Let $C$ be a curve in $\proj{3}$ with
  Rao module $M(C)$, and suppose there is a graded $R$-module isomorphism $M(C)
  \cong M' \oplus M_{[t]}$. If $F_1 \cong Q_1 \oplus R(-t)$ in the minimal
  resolution  \eqref{resoluMt}
  of the homogeneous ideal $I(C)$:
\begin{equation} \label{resoluMI40}
  0 \rightarrow P_4 \oplus R(-t-4)  \xrightarrow{(\sigma', 
  \sigma_{[t]}) \oplus 0} P_3 \oplus
  R(-t-3)^{4} \oplus 
  F_2  \rightarrow F_1 \rightarrow I(C) \rightarrow 0 \, , \ 
\end{equation}
and if $P_2$ does not contain a direct summand $R(-t)$ (i.e.
$\beta_{3,t}(M')=0$, cf.\;\eqref{MplussKosz}), then there is a generization $C'
\subseteq \proj{3}$ of $C \subseteq \proj{3}$ in $\HH(d,g)$ with constant
specialization and constant $M'$ (up to a graded $R$-module isomorphism) such
that $I(C')$ has the $R$-free resolution:
\begin{equation}  \label{resoluminMI40}
  0 \rightarrow P_4  \xrightarrow{\sigma' \oplus 0 \oplus 0} P_3 \oplus F_2
  \oplus R(-t-2)^{6} 
  \rightarrow Q_1 \oplus R(-t-1)^{4}  \rightarrow I(C') \rightarrow 0 \, . 
\end{equation}
The resolution is minimal except possibly in degree $t+1$ and $t+2$ in which
some of the summands of $R(-t-1)^{4}$ (resp. $R(-t-2)^{6}$) may be cancelled
against corresponding free summands of $F_2$ (resp. $Q_1$). Moreover there
exists a generization as above with a minimal resolution where all free common
summands of $\{F_2,R(-t-1)^{4}\}$ and $\{R(-t-2)^{6}, Q_1\}$ are cancelled.
\end{proposition}

The idea of a proof is to link $C$ to a curve $D$ by a CI of type $(f,g)$
where $f \ne t$ and $g \ne t$, then to take a generization of $D$ by using
Proposition~\ref{mainres1} because the degree-$t$ generator of $I(C)$ leads to
a ghost term for $D$ exactly where it appears in Proposition~\ref{mainres1}.
Finally we link back via a CI of the same type $(f,g)$ as before. Since there
are some technical challenges involved,
we give an example which, to a
certain extent, illustrate the proof.

\begin{example} \label{exe10} Take the minimal resolution of a smooth
  Buchsbaum curve $C$ of degree $6$ and genus $3$:
   \begin{equation*} 0 \rightarrow R(-6) \rightarrow R(-5)^{4}
    \rightarrow R(-4)^3 \oplus R(-2) \rightarrow I(C) \rightarrow 0 \, .
    \end{equation*} 
    It has the form as in the resolution of $I(C)$ in
    Proposition~\ref{newmainres} with $M'=0$ (and hence all $P_i=0$) and
    $t=2$. We claim there is a generization ``cancelling the leftmost term
    $R(-6)$ (together with $R(-5)^4$) against $R(-2)$'' at the cost of an
    increase in Betti numbers in degrees $3$ and
    $4$. 
    To see it we link $C$ to $D$ via a CI of type $(f,g)$ containing $C$. We
    take $f=g=4$ to simplify, but the argument works for any CI avoiding the
    quadric. 
    Let
    \begin{equation} \label{Et} E_t:=\coker \sigma_{[t]} \qquad where \qquad
      \sigma_{[t]}:= R(-t-4) \longrightarrow R(-t-3)^4
    \end{equation}  be given by the exact 
    sequence \eqref{Kosz}. That sequence also give the exactness of
    \begin{equation} \label{Etv} 0 \rightarrow R(t)
      \xrightarrow{\tau_{[t]}^{\vee}} R(t+1)^{4} \rightarrow R(t+2)^6
      \rightarrow E_t^{\vee} \rightarrow 0 \, .
    \end{equation} 
    The $E$-resolution of $I(C)$ is \ $ 0 \rightarrow  E_2
    \rightarrow R(-4)^3 \oplus R(-2) \rightarrow I(C) \rightarrow 0$, which
    through \eqref{resoluLink} yields
    \begin{equation} 0 \rightarrow R(-6) \oplus R(-4) \rightarrow
      E_2^{\vee}(-8) \rightarrow I(D) \rightarrow 0 \, 
    \end{equation} 
    by removing 2 redundant terms. Using \eqref{Etv} and the mapping cone
    construction as in \eqref{resoluFreeLink}, we get: 
    \begin{equation*} 0 \rightarrow R(-6) \rightarrow R(-5)^{4} \oplus R(-6)
      \rightarrow R(-4)^{5} \rightarrow I(D) \rightarrow 0 \, .
    \end{equation*} 
    This resolution has the form as in Proposition~\ref{mainres1} with $M'=0$
    and $t=2$. By that Proposition there is a generization $D'$ cancelling the
    ghost term $R(-6)$, and we get an ACM curve. Finally we link ``back'' via a
    general CI of type $(4,4)$, and we get a curve $C'$ with minimal
    resolution,
    \begin{equation*} 0 \rightarrow R(-4)^{3} \rightarrow R(-3)^{4}
      \rightarrow I(C') \rightarrow 0 \,
    \end{equation*} 
    which, thanks to \cite{K3}, Prop. 3.7, is a generization of the original
    curve $C$.
\end{example} 
Since we certainly do not want to have Proposition~\ref{newmainres} only for
curves whose Rao module $M(C)$ is a one-dimensional $k$-vector-space, we
consider curves with a Buchsbaum component in Proposition~\ref{newmainres},
making any diameter-1 curve the special case $M(C) \cong M_{[t]}^r$.

\begin{proof} [Proof (of Proposition~\ref{newmainres})] First we find the
  $E$-resolution of $I(C)$. Using \eqref{resoluMI} and \eqref{resoluE} and the
  notations from \eqref{MplussKosz}--\eqref{resoluMt}, we get the
  $E$-resolution
  \begin{equation} \label{resoluEEt} 0 \rightarrow E \oplus E_t \oplus F_2
    \rightarrow F_1 \rightarrow I(C) \rightarrow 0 \ , \qquad E:= \coker
   ( P_4  \xrightarrow{\sigma'} P_3 )
\end{equation}
where $E_t$ is given by \eqref{Et}. Now linking $C$ to $D$ via a CI of type
$(f,g)$, $f,g \gg 0$, the resolution \eqref{resoluLink} of $I(D)$ is given
by
\begin{equation} \label{middleT} 0 \rightarrow F_1^{\vee} \rightarrow E^{\vee}
  \oplus E_t^{\vee} \oplus F_2^{\vee} \oplus R(f) \oplus R(g) \rightarrow
  I(D)(f+g) \rightarrow 0 \, .
\end{equation}
The exact sequences \eqref{Etv} and $ 0 \to P_0^{\vee}
\xrightarrow{\tau_1^{\vee}} P_1^{\vee} \xrightarrow{\tau_2^{\vee}} P_2^{\vee}
\to E^{\vee}\to 0 $ yield an $R$-free resolution of the middle term of
\eqref{middleT}, which through the mapping cone construction as in
\eqref{resoluFreeLink} implies an $R$-free resolution:
\begin{equation} \label{mapsMCC} 0 \rightarrow P_0^{\vee} \oplus R(t)
  \xrightarrow{(\tau_1^{\vee},\tau_{[t]}^{\vee}) \oplus 0} P_1^{\vee} \oplus
  R(t+1)^4 \oplus F_1^{\vee} \xrightarrow{ \ \beta} F_1' \to I(D)(f+g)
  \rightarrow 0 \, 
\end{equation}
where $ F_1':=P_2^{\vee} \oplus R(t+2)^6 \oplus F_2^{\vee} \oplus R(f) \oplus
R(g)$, noticing that the morphism $ F_1^{\vee} \to P_2^{\vee} \oplus R(t+2)^6
$ corresponding to a submatrix of $\beta$ may be non-minimal because we in the
mapping cone construction need to lift the morphism $ F_1^{\vee} \to E^{\vee}
\oplus E_t^{\vee}$ to $ F_1^{\vee} \to P_2^{\vee} \oplus R(t+2)^6 $. Note also
that the mapping cone construction allows us to take the morphisms $
P_1^{\vee} \to F_1'$ deduced from $\beta$ (resp. the leftmost in
\eqref{mapsMCC}) as $\tau_2^{\vee} \oplus 0 \oplus 0 \oplus 0 \oplus 0$ (resp.
$(\tau_1^{\vee},\tau_{[t]}^{\vee}) \oplus 0$), see \cite{W}. The resolution
\eqref{mapsMCC} has the form as in Proposition~\ref{mainres1} because $
F_1^{\vee} = Q_1^{\vee} \oplus R(t)$. Hence there is a generization $D'$
cancelling the ghost term $R(t-f-g)$ from the resolution of $I(D')$ because
Remark~\ref{mainres2corr} allows to use Proposition~\ref{mainres1} for
non-minimal resolutions in the case $P_2^{\vee}$ does not contain $R(t)$. We
get (where now the induced $ Q_1^{\vee} \to P_2^{\vee} \oplus R(t+2)^6 $ may
be non-minimal):
\begin{equation} 0 \rightarrow P_0^{\vee} \xrightarrow{\tau_1^{\vee} \oplus 0
    \oplus 0} P_1^{\vee} \oplus R(t+1)^4 \oplus Q_1^{\vee}
  \xrightarrow{\ \alpha} F_1' \to I(D')(f+g) \rightarrow 0 \, .
\end{equation}
In addition the morphism $ R(t+1)^4 \rightarrow F_1' $ corresponding to a
submatrix of $\alpha$ may be non-minimal by Remark~\ref{mainres2corr}. Letting
$E_{\tau_1}:=\ker(P_1 \xrightarrow{\tau_1} P_0)$, 
then an $E$-resolution is
\begin{equation} 0 \rightarrow E_{\tau_1}^{\vee} \oplus
  R(t+1)^4 \oplus Q_1^{\vee} \rightarrow F_1' \to I(D')(f+g) \rightarrow 0 \, .
\end{equation}
Since $D'$ is a generization of $D$ with constant postulation, there is a
generization $Y' \supset D'$ of $Y$ of type $(f,g)$, such that the linked
curve $C'$ is a generization of $C$, cf. \cite{K3}, Prop.\;3.7 (the
assumptions of Prop.\;3.7 are weak, and they are at least satisfied if
$H^1(\sI_C(v)) = 0$ for $v = f, g, f-4$ and $g - 4$, which we may assume by
$f,g \gg 0$). Using \eqref{resoluLink}, we get the resolution
\begin{equation} 0 \rightarrow  F_1'^{\vee} \to E_{\tau_1} \oplus
  R(-t-1)^4 \oplus Q_1 \oplus R(-f) \oplus R(-g) \to I(C') \rightarrow 0 \, .
\end{equation}
Noting that $0 \to P_4 \xrightarrow{\sigma'} P_3 \to P_2 \to E_{\tau_1} \to 0$
is exact and letting the lifting of $ F_1'^{\vee} \to E_{\tau_1}$ to $
F_1'^{\vee} \to P_2$ be the natural one (the form of $ P_1^{\vee} \to F_1'$
above allows us to take the dual of $ F_1'^{\vee} \to P_2$ as $id \oplus 0
\oplus 0 \oplus 0 \oplus 0$, $id$ the identity), the mapping cone construction
yields (cf.\;\eqref{resoluFreeLink})
\begin{equation} 0 \rightarrow P_4 \xrightarrow{\sigma' \oplus 0}  P_3 \oplus
  F_1'^{\vee} \to P_2 \oplus R(-t-1)^4 \oplus Q_1 \oplus R(-f) \oplus
  R(-g) \to I(C') \rightarrow 0 \, .
\end{equation} If we now replace $ F_1'$ with its defining expression, we get
exactly the resolution of the proposition provided we can show that the
repeated free summand $P_2 \oplus R(-f) \oplus R(-g)$ is redundant. This is
obvious for $P_2$. Note that in the resolution where  $P_2$ is deleted, the
possibly non-minimality of $ Q_1^{\vee} \to P_2^{\vee} \oplus
R(t+2)^6 $ reduces to a possibly non-minimality of $ Q_1^{\vee} \to 
R(t+2)^6 $ and moreover, ghost terms between $Q_1$
and $F_2$ remain ghost terms (easily seen from  the form of  $ F_1'^{\vee}
\to P_2$ above). Finally even though it is rather easy to see
that $R(-f) \oplus R(-g)$ is redundant because  $f,g \gg 0$, we choose instead
to use the idea in 
the proof of Theorem~\ref{mainres} which imply  
that this free summand  becomes at least redundant after a 
generization (and no ghost terms between $Q_1$
and $F_2$ become redundant), whence we get the
desired $R$-free resolution. We also get the 
minimality of the resolution in degree $\ne t+1,t+2$ by observing that in
this proof, there are eventually only two places where the resolution may be
non-minimal, namely for the above mentioned morphisms  $ Q_1^{\vee} \to 
R(t+2)^6 $ and  $R(t+1)^4 \rightarrow F_2^{\vee}$. Since we get the final
statement from  Theorem~\ref{mainres}, we are done.
\end{proof}

\begin{corollary} \label{mainres2} Let $M(C) \cong M' \oplus M_{[t]}$ as
  graded $R$-modules, let $(a_1,a_2,b_1,b_2)$ be the corresponding 4-tuple and
  suppose $\beta_{3,t}(M')=0$. If $a_2 \neq 0$ (recall $a_2:=\beta_{1,t}$),
  then there is a generization $C'$ of $C$ in $\HH(d,g)$ with constant
  specialization and $M'$ whose 4-tuple
  is \\[2mm]
  {\rm (P2)} \hspace{5cm} $(a_1,a_2-1,b_1,b_2) \ .$ \\[2mm] Moreover for $i
  \in \{0,1\}$, $h^i(\sI_{C'}(v)) = h^i(\sI_{C}(v))$ for $v \neq t$ and
  $h^i(\sI_{C'}(t)) = h^i(\sI_{C}(t))-1$.
\end{corollary}

\begin{remark} Strictly speaking we need an extension of the notion of a
  4-tuple for the generization $C'$ of $C$ because $ M_{[t]}$ disappear for
  $C'$ (e.g. $C'$ may be ACM). We have, however, the number $t$ attached to
  $C$ and so it is clear which Betti numbers decrease.
    \end{remark}

\section{The graded Betti numbers of diameter-1 curves}
Since our results become quite complete for a diameter-$1$ (Buchsbaum) curve
$C \subseteq \proj{3}$, we now consider such curves closely. The main result
of this section describes ``all'' generizations of a diameter-1 curve $C$ in
$\HH(d,g)$, from the point of view of describing their minimal resolutions. In
other word, we give essentially all possible choices of the graded Betti
numbers of a generization of a diameter-1 curve. In particular we determine
the form of the minimal resolutions of all generic curves of the irreducible
components of $\HH(d,g)$ that contain $(C)$ and we find how many such
components exist. Note that these results somehow complete works of
Chang (\cite{C}, Ex.\;1, \cite{W2}, Thm.\;4.1, \cite{W1}) which, to a large
degree, determine the set of graded Betti numbers for which there exists (even
smooth connected) diameter-1 curves.

For a diameter-$1$ curve $C \subseteq \proj{3}$, we have $M(C) \cong
M_{[t]}^r$ with $t=c$, and a 5-tuple $(a_1,a_2,b_1,b_2,r) =
(\beta_{1,c+4},\beta_{1,c}, \beta_{2,c+4}, \beta_{2,c}, \beta_{3,c+4})$. The
minimal resolution (the Rao form) of $I(C)$ is
\begin{equation} \label{resoluMIR} 0 \rightarrow R(-c-4)^{r} \xrightarrow{\
    \sigma_{[c]} \oplus 0 \ } R(-c-3)^{4r} \oplus F_2 \rightarrow F_1
  \rightarrow I(C) \rightarrow 0 \, .
    \end{equation}

  \begin{remark} \label{mainres4} Suppose $\diam M = 1$, i.e. $M(C)
  \cong M_{[c]}^r$ and let $\beta_{j,i}:=\beta_{j,i}(C)$.

  {\rm (a)} By Remark~\ref{mainres2corr} there is a generization given by {\rm
    (P1)}, see Corollary~\ref{mainres12}, whose graded Betti numbers do not
  change except for $\beta_{3,c+4}$ and $\beta_{2,c+4}$, which both decrease
  by $1$, and $\beta_{1,c+3}$ and $\beta_{2,c+3}$, which may decrease by at most
  $4$, keeping, however, $\beta_{1,c+3}- \beta_{2,c+3}$ unchanged. Moreover if
  we combine with Theorem~\ref{mainres}, we may suppose that $\beta_{2,c+3}$
  decreases by exactly $ \min \{ \beta_{1,c+3},4\}$ after possibly further
  generizations (i.e. using  {\rm (Q(c$\,+\,3$))} of Corollary~\ref{mainrecor}).

  {\rm (b)} By Proposition~\ref{newmainres} we can describe the possible
  changes of the graded Betti numbers of a generization given by {\rm (P2)} in
  detail. Indeed the graded Betti numbers of a generization as in
  Corollary~\ref{mainres2} do not change except $\beta_{3,c+4}$ and
  $\beta_{1,c}$, which both decrease by $1$, and $\beta_{1,v}$ and $\beta_{2,v}$
  for $v \in \{c+3,c+2,c+1\}$ for which $\beta_{1,c+1}- \beta_{2,c+1}$
  increases by $4$, $\beta_{2,c+2}- \beta_{1,c+2}$ increases by $6$ and
  $\beta_{2,c+3}$ decreases by $4$. Moreover combining with
  Theorem~\ref{mainres}, we may suppose that $\beta_{1,c+1}$ increases by $4-
  \min \{ \beta_{2,c+1},4\}$ and $\beta_{2,c+1}$ decreases by $\min \{
  \beta_{2,c+1},4\}$ while $\beta_{2,c+2}$ increases by $6- \min \{
  \beta_{1,c+2},6\}$ and $\beta_{1,c+2}$ decreases by $\min \{
  \beta_{1,c+2},6\}$ after possibly further generizations.

  {\rm (c)} Combining {\rm (a)} and {\rm (b)} by mainly using {\rm (Pi)} $p_i$
  times for $i=1,2$, we get the existence of a generization $C'$ of $C$ in
  $\HH(d,g)$ whose Betti numbers  $ \{ \beta'_{j,i}\}$ satisfy: \\[-3mm]

  \quad  $\beta'_{1,c+4} = \beta_{1,c+4}$\,, \ \ $\beta'_{2,c+4} =
  \beta_{2,c+4}-p_1$\,, \quad  \ $\beta'_{3,c+4} = \beta_{3,c+4}-p_1-p_2$\,,

\quad $\beta'_{1,c+3} = \beta_{1,c+3}-\min\{4p_1, \beta_{1,c+3}\}$\,,
  \qquad \quad $\beta'_{2,c+3} = \beta_{2,c+3}-4p_2-\min\{4p_1,
  \beta_{1,c+3}\}$,

  \quad $\beta'_{1,c+2} = \beta_{1,c+2}-\min\{6p_2, \beta_{1,c+2}\}$\,,
  \qquad \quad $\beta'_{2,c+2} = \beta_{2,c+2}+6p_2-\min\{6p_2,
  \beta_{1,c+2}\}$,

  \quad $\beta'_{1,c+1} = \beta_{1,c+1}+4p_2-\min\{4p_2, \beta_{2,c+1}\}$\,,
  \  $\beta'_{2,c+1} = \beta_{2,c+1}-\min\{4p_2, \beta_{2,c+1}\}$,

  \quad $\beta'_{1,c} = \beta_{1,c}-p_2$\,, \ $\beta'_{2,c} = \beta_{2,c}$,
  \qquad \qquad \quad \,\ $\beta'_{j,i} = \beta_{j,i}$\, \ for $j=1,2$ and \
  every
  $i \notin B$\,, \\[1mm]
  where $B= \{c,c+1,c+2,c+3,c+4\}$. In particular the 5-tuple of $C'$ is
  \\[-3mm]
  \[ (\beta_{1,c+4},\beta_{1,c}-p_2, \beta_{2,c+4}-p_1, \beta_{2,c},
  \beta_{3,c+4}-p_1-p_2)\,.\]
\end{remark}

Now we come to the main theorems of the paper. But first we need a definition.

\begin{definition} \label{defrep} Let $C$ be a diameter-$1$ curve in
  $\proj{3}$, $(C) \in \HH(d,g)$, and let $J$ be a subset of the natural
  numbers $ \mathbb N$. Then a generization $C'$ of $C$ in $\HH(d,g)$ that is
  given by repeatedly using some of the generizations furnished by {\rm (P1)},
  {\rm (P2)} and {\rm (Qj)} for $j \in J$ in some order, is called a
  generization in $\HH(d,g)$ generated by ${\rm (PQ_J)}$. If only {\rm (Qj)},
  $j \in J$ is used, we call it a generization generated by ${\rm (Q_J)}$. We
  omit the index $J$ in ${\rm (PQ_J)}$ and ${\rm (Q_J)}$ in the case $J=
  \mathbb N$. Moreover we allow $J = \emptyset$ in the definitions, in which
  case $C'$ is a trivial generization of $C$. 
\end{definition}

Since the generizations given by (Pi) and (Qj) composed with a trivial
generization (Definition~\ref{genT}), is again a generization given by (Pi)
and (Qj) respectively, we get that e.g. a generization $C'$ of $C$ generated
by ${\rm (PQ_J)}$ is, up to a trivial generization, independent of the order
in which we use (Pi) and (Qj). Indeed if we change the order we still get a
generization $C''$ of $C$ in $\HH(d,g)$ in which $C''$ and $C'$ belong to the
same Betti stratum, and we conclude from Corollary~\ref{gencorT}.

Now we can prove that any generization of $C$ in $\HH(d,g)$ is generated by
${\rm (PQ)}$, up to the removal of some ghost terms between $F_2$ and $F_1$ in
the degrees $c+1,c+2,c+3$ of \eqref{resoluMIR}.
\begin{theorem} \label{newmain} Let $C \subseteq \proj{3}$ be a Buchsbaum
  curve of diameter one and let $C'$ be any generization of $C$ in $\HH(d,g)$.
  If $A= \{c+1,c+2,c+3\}$ then there exists a generization $C''$ of $C'$
  generated by ${\rm (Q_A)}$ such that $C''$ is a generization of $C$ in
  $\HH(d,g)$ generated by ${\rm (PQ)}$.
\end{theorem}

The proof relies on the following semi-continuity result:
\begin{proposition} \label{newmainprop} Let $C$ be a Buchsbaum curve in
  $\proj{3}$ of diameter one. If $v \notin \{c+1,c+2,c+3 \}$, then the Betti
  numbers $ \beta_{1,v}$ and $ \beta_{2,v}$ are upper semi-continuous. In
  particular the 5-tuple $( \beta_{1,c+4},
  \beta_{1,c},\beta_{2,c+4},\beta_{2,c},\beta_{3,c+4})$ is upper
  semi-continuous, i.e. each of these 5 numbers do not increase under
  generization.
\end{proposition}
\begin{remark} \label{newmainproprem} If $C$ is ACM, then the Betti numbers $
  \beta_{1,v}$ and $ \beta_{2,v}$ are upper semi-continuous for any integer v.
  This is well known, but the result also follows from Corollary~\ref{semic}.
\end{remark}
\begin{proof} We will prove the result by using so-called $\Omega$-resolutions
  of a Buchsbaum curve (\cite{C}, it is really the dual of an $E$-resolution
  involving $M_{[t]}$ for $t=0$). Recall that $\Omega$ is by definition given
  by the exact sequences
  \begin{equation} \label{newex11} 0 \rightarrow \widetilde {\Omega}
    \rightarrow \sO_{\proj{}}(-1)^{4} \rightarrow \sO_{\proj{}} \rightarrow 0\
    \ {\rm and} \ \ 0 \rightarrow R(-4) \rightarrow
    R(-3)^{4} \rightarrow R(-2)^{6} \rightarrow \Omega
    \rightarrow 0
\end{equation}
which we deduce from the Koszul resolution of the regular sequence $\{
X_0,X_1,X_2,X_3\} $, whence
\begin{equation} \label{newexom}
 H_{*}^2(\widetilde \Omega)= 0 \ , H^1(\widetilde \Omega (0))
\simeq k \ , \ H^1(\widetilde \Omega(v))= 0
\ {\rm for} \ v \neq 0 \ .
\end{equation}
Note also that $\widetilde \Omega(2)$ is 0-regular and generated by global
sections. It follows that if we tensor the $1^{st}$ exact sequence of
\eqref{newex11} by $\widetilde \Omega(v)$ and take cohomology, we get
\begin{equation} \label{newexom2}
 H^1(\widetilde \Omega^{\otimes 2} (v)) = 0
\ \ \ {\rm for} \ v \neq 1 \ {\rm and} \ 2 \ .
\end{equation}
Since $r= h^1(\sI_C(c))$, the $\Omega$-resolution of $C$ of
Proposition~\ref{newmainprop}, twisted by $c$, is given
by 
\begin{equation} \label{newex13} \ \ 0 \rightarrow G_2 \rightarrow \Omega^r
  \oplus G_1 \rightarrow I(C)(c) \rightarrow 0
\end{equation}
where $ \ G_i$ for $i=1,2$ is free and the induced map $G_2 \to G_1$ is
minimal. Using that a minimal resolution of $ \Omega^r$ is just a direct sum
of the resolution given in \eqref{newex11}, we get by the mapping cone
construction the following free resolution of $I(C)(c)$ (cf. the proof of
Proposition~\ref{newmainres})
\begin{equation} \label{newex013}
    0 \rightarrow R(-4)^{r} \xrightarrow{\ \sigma \oplus
      0 \ } R(-3)^{4r} \oplus G_2 \rightarrow  R(-2)^{6r} \oplus G_1
      \rightarrow I(C)(c) \rightarrow 0 \  \
\end{equation}
that is minimal except possibly in degree 2 and 3. 
Comparing we see that $G_j(-c)$, $j=1,2$, contains exactly the free summand $
R(-i)^{\beta_{j,i}}$ of degree $i$ for $i \notin \{2,3 \}$. We {\it claim}
that
\begin{equation} \label{newex12} h^1(\sI_C \otimes \widetilde {
  \Omega}(v)) = \beta_{1,v} \ , \quad {\rm for} \ \ v \notin
\{c+1,c+2,c+3 \}.
\end{equation}
To prove it we sheafify \eqref{newex13} and tensor with $ \widetilde {
  \Omega}(v-c)$. Since $ H_{*}^2(\widetilde \Omega)= 0$ and
$H^1(\widetilde {\Omega}^{\otimes 2}(v-c))=0$, it follows that the sequence
\begin{equation} \label{newex14} H^1 (\widetilde {G_2}(-c) \otimes
  \widetilde {\Omega}(v)) \to H^1(\widetilde {G_1}(-c) \otimes \widetilde
  {\Omega}(v)) \rightarrow H^1(\sI_C \otimes \widetilde { \Omega}(v))
  \rightarrow 0
\end{equation}
is exact. Due to \eqref{newexom} the sequence \eqref{newex14}
yields $ H^1 ( \widetilde {\Omega}^{ \beta_{2,v}}) \to
H^1 ( \widetilde {\Omega}^{ \beta_{1,v}}) \rightarrow H^1(\sI_C \otimes
\widetilde { \Omega}(v)) \rightarrow 0$. 
By the minimality of $G_2 \to G_1$, we deduce the equality in
\eqref{newex12}. 

Using the proven claim, we get that each of the $ \beta_{1,v}$ is
semi-continuous since $h^1(\sI_C \otimes \widetilde { \Omega}(v))$ is
semi-continuous. To see the corresponding statement for $ \beta_{2,v}$, we use
again linkage. Note that if we link $C$ to $D$ via a CI of type $(f,g)$, we
get $c(D)=f+g-4-c$, and
 $$ \beta_{2,v}(C)= \beta_{1,c+c(D) +4-v}(D)
 \ , \quad {\rm for} \ \ v \notin \{c+1,c+2,c+3 \}.$$
 by \eqref{resoluFreeLink}. By \eqref{newex12} we get that $ \beta_{1,c+c(D)
   +4-v}(D)$ is semi-continuous, because $v \notin \{c+1,c+2,c+3 \}$ is
 equivalent to $c+c(D) +4-v \notin \{ c(D)+1,c(D)+2,c(D)+3 \}$. Finally since
 $r= h^1(\sI_C(c))$ is clearly semi-continuous, we get the semi-continuity for
 every $ \beta_{i,v}$ of Proposition~\ref{newmainprop}, as well as for the
 5-tuple of graded Betti numbers, and we are done.
\end{proof}

\begin{proof}[Proof (of Theorem~\ref{newmain})]
  We denote the 5-tuple $( \beta_{1,c+4},
  \beta_{1,c},\beta_{2,c+4},\beta_{2,c},\beta_{3,c+4})$ of $C$ shortly by
  $(a_1,a_2,b_1,b_2,r)$ and there is a corresponding 5-tuple,
  $(a_1',a_2',b_1',b_2',r')$ for the generization $C'$. We write the 5-tuple
  of the operations (P1) as $(a_1,a_2,b_1-1,b_2,r-1)$ and (P2) as
  $(a_1,a_2-1,b_1,b_2,r-1)$. Repeated use of (Pq) for $q=1,2$ implies the
  existence of a generization of $C$ with 5-tuple
  $(a_1,a_2-i,b_1-j,b_2,r-i-j)$ provided $a_2-i \ge 0$, $b_1-j \ge 0$ and
  $r-i-j \ge 0$. Recalling $\gamma_{C}(v)= h^0(\sI_{C}(v))$ and $\sigma_C(v)=
  h^1(\sO_{C}(v))$, we {\it claim } that
\begin{equation} \label{ine}
  h^0(\sI_{C'}(c)) - a_2' \ge  h^0(\sI_{C}(c)) - a_2   \ \ {\rm and} \ \
  h^1(\sO_{C'}(c))-  b_1' \ge  h^1(\sO_{C}(c)) - b_1 \, .   
\end{equation}
We only prove the first inequality since the latter is the "dual" result which
one may get from the first inequality by linkage. To prove it we remark that
$\gamma_{C'}(v) = \gamma_{C}(v)$ for $v < c$ by the semi-continuity of
$\gamma_{C}(v)$ and $\sigma_{C}(v)$ because $\chi(\sI_{C'}(v)) =
\chi(\sI_{C}(v))$ implies 
\begin{equation} \label{ksiv}
\gamma_{C}(v)+\sigma_C(v) = 
\gamma_{C'}(v)+\sigma_{C'}(v) \ \ {\rm for} \ v \neq c \, .
\end{equation}  Using the exactness of the
minimal resolutions of $\sI_{C'}$ and $\sI_{C}$ in degree $v=c$, we get that
$h^0(\sI_{C'}(c)) - a_2' + b_2' = h^0(\sI_{C}(c)) - a_2 + b_2$ since the
exactness of these resolutions in degree $v<c$ implies $ \beta_{1,v}(C) -
\beta_{2,v}(C)= \beta_{1,v}(C') - \beta_{2,v}(C')$ for every $v<c$. Since we
know $b_2' \le b_2$ by the semi-continuity of Proposition~\ref{newmainprop}, we
get $\gamma_{C'}(c) - a_2' \ge \gamma_{C}(c) - a_2$, i.e. the claim.

Now let $\Delta\gamma(c):=\gamma_{C}(c) - \gamma_{C'}(c)$. Using
Corollary~\ref{mainres2} $\Delta\gamma(c)$ times, we get the existence of a
generization $C_{P2}$, furnished by (P2), with constant specialization
($\sigma_{C_{P2}}(v)=\sigma_C(v)$) and with the same postulation as $C'$.
Indeed this is possible because $a_2 \ge \Delta\gamma(c) \ge 0$ by \eqref{ine}
and $r \ge \Delta\gamma(c)$ by the semi-continuity of $ h^1(\sO_{C}(c))$ that
implies $$ h^0(\sI_{C'}(c))-r' = \chi(\sI_{C'}(c))-h^1(\sO_{C'}(c)) \ge
\chi(\sI_{C}(c))-h^1(\sO_{C}(c)) = h^0(\sI_{C}(c))-r . $$

Next we use Corollary~\ref{mainres12} $\Delta\sigma(c):=\sigma_{C}(c) -
\sigma_{C'}(c)$ times to get the existence of a generization $C_{P}$ of
$C_{P2}$, furnished by (P1), with constant postulation ($\gamma_{C_P}(c)=
\gamma_{C_{P2}}(v)$) and with the same specialization as $C'$. This is
possible because $b_1 \ge \Delta\sigma(c) \ge 0$ by \eqref{ine} and $r -
\Delta\gamma(c) \ge \Delta\sigma(c)$. Indeed the latter follows at once from
the equality $\chi(\sI_{C'}(c))=\chi(\sI_{C}(c))$ that implies $r - r' =
\Delta\gamma(c) + \Delta\sigma(c)$.

So far we have two curves $C_P$ and $C'$ that by \eqref{ksiv} and the
construction of $C_P$ have the same postulation and specialization
functions, whence $ h^1(\sI_{C'}(c))= h^1(\sI_{C_P}(c))$. It follows that
$\beta_{3,v}(C') = \beta_{3,v}(C_P)$ for $v=c+4$ and hence for every $v$.
Since $\gamma_{C'}=\gamma_{C_P}$, we get
\begin{equation} \label{bettieq} 
\beta_{1,v}(C')-\beta_{2,v}(C') =\beta_{1,v}(C_P)-\beta_{2,v}(C_P) \
\end{equation} 
for every $v$ by \cite{P}. We {\it claim} that $\beta_{i,j}(C') \le \beta_{i
  ,j}(C_P)$ for $i=1,2$ and $j \notin A$. First take $ j \notin \{c,c+4\} \cup
A$. Then $\beta_{i,j}(C)=\beta_{i,j}(C_P)$ by the construction of $C_P$ and
$\beta_{i,j}(C')\le \beta_{i,j}(C)$ by Proposition~\ref{newmainprop}, and we
get the claim. Next we consider $j=c$. Then
$\beta_{1,c}(C_P)=\beta_{1,c}(C)-\Delta \gamma(c)$ and
$\beta_{2,c}(C_P)=\beta_{2,c}(C)$ by the construction of $C_P$ or by
Remark~\ref{mainres4} (c). Since $\beta_{1,c}(C') =a_2' \le
\beta_{1,c}(C)-\Delta \gamma(c)$ by \eqref{ine} and $\beta_{2,c}(C')\le
\beta_{2,c}(C)$ by Proposition~\ref{newmainprop}, we get the claim for $j=c$.
Finally for $j=c+4$ we use the other inequality of \eqref{ine},
Remark~\ref{mainres4} (c) and Proposition~\ref{newmainprop} to see
$\beta_{i,c+4}(C') \le \beta_{i ,c+4}(C_P)$ for $i=1,2$, and the claim is
proved.

If the inequality of the claim is strict for some $j \notin A$ and some $i \in
\{1,2\}$, then both $\beta_{1,j}(C_P)$ and $\beta_{2,j}(C_P)$ are non-zero by
their semi-continuity and \eqref{bettieq}, and $R(-j)$ is a common free
summand of $F_2$ and $F_1$ in the minimal resolution of $I(C_P)$. Hence
Theorem~\ref{mainres} applies to $R(-j)$ as well as to any other ghost term
between $F_2$ and $F_1$ in the minimal resolution of $I(C_P)$ for which the
inequality of the claim is strict. It follows that there is a generization $D$
of $C_P$ generated by $(Q_{\mathbb N -A})$ such that
$\beta_{i,j}(C')=\beta_{i,j}(D)$ for $i=1,2$ and $ j \notin A$.

Finally if $j \in A$, we still have \eqref{bettieq}. It follows that we either
have $\beta_{i,j}(C') = \beta_{i ,j}(C_P)$ for $i=1,2$, {\it or}
$\beta_{i,j}(C') < \beta_{i ,j}(C_P)$ for $i=1,2$, whose corresponding ghost
term in the minimal resolution of $I(C_P)$ is removed by a generization of
$D$, {\it or} $\beta_{i,j}(C') > \beta_{i ,j}(C_P)$ for $i=1,2$, leading to a
ghost term in the minimal resolution of $I(C')$ that is removed by a
generization given by (Qj) of $C'$. Removing all such ghost terms
corresponding to strict inequalities of the graded Betti numbers above, we get
the existence of generizations $C''_1$ of $C'$, and $D'$ of $D$, generated by
${\rm (Q_A)}$ such that $\beta_{i,j}(C''_1) = \beta_{i ,j}(D')$ for every $i
\in \{1,2\}$ and $j \in \mathbb N$. Since (Qj) do not change $\beta_{3,c+4}$,
then the generizations $C''_1$ and $D'$ of $C$ belong to the same Betti
stratum. Using 
Corollary~\ref{gencorT} and Definition~\ref{genT}, we get the theorem.
\end{proof}

\begin{remark} \label{remmainthm4} Let $C'$ be a generic curve of the Betti
  stratum 
  of a diameter-1  curve $C$. Then it follows from the last paragraph
  of the proof that if $\beta_{i,j}(C') \le \beta_{i ,j}'$ for $i=1,2$ and $j
  \in A$ where $\beta_{i ,j}'$ is given as in Remark~\ref{mainres4} (c), we
  may take $C''= C'$ in Theorem~\ref{newmain}, i.e. $C'$ is a generization of
  $C$ in $\HH(d,g)$ generated by ${\rm (PQ)}$.
\end{remark}



A main application of Theorem~\ref{newmain} is the
first statement (``the hard part'') of the following: 

\begin{corollary} \label{cormaxcomp} Let $C'$ be a generic curve of an
  irreducible component of $\HH(d,g)$ containing a diameter-1 curve $C$, and
  let $c=c(C)$ and $\beta'_{i,j}=\beta_{i,j}(C')$. Then $C'$ is a generization
  of $C$ in $\HH(d,g)$ generated by ${\rm (PQ)}$. Moreover \ $\beta'_{2,c+4}
  \cdot \beta'_{3,c+4}=  \beta'_{1,c} \cdot \beta'_{3,c+4}= 0\,, $ 
  \[ \ \beta'_{1,c+3} \cdot (\beta'_{2,c+3}-4 \beta'_{3,c+4})=0 \ , \ \
 \beta'_{1,i} \cdot \beta'_{2,i} = 0  \ \ {\rm for \ any \ } i \ne c+3 \,
, 
 \]
 and its 5-tuple is either $(\beta'_{1,c+4},0,0,\beta'_{2,c},\beta'_{3,c+4})$
 with $\beta'_{3,c+4} \ne 0$ or
 $(\beta'_{1,c+4},\beta'_{1,c},\beta'_{2,c+4},\beta'_{2,c},0)$. 
 %
\end{corollary}

\begin{proof} The generic curve $C'$ is a generization of $C$ in $\HH(d,g)$,
  whence is generated by ${\rm (PQ)}$ by Theorem~\ref{newmain} or
  Remark~\ref{remmainthm4}. Moreover the generic curve $C'$ must satisfy $ \
  r' \cdot \beta'_{1,c}= 0 \ {\rm and } \ r' \cdot \beta'_{2,c+4} = 0$ where
  $r'=\beta'_{3,c+4}$, because otherwise there exists by
  Corollaries~\ref{mainres12} and \ref{mainres2} a generization $C''$ of $C'$
  such that $\beta_{3,c+4}(C'') = \beta'_{3,c+4}-1$ contradicting the
  semi-continuity of $\beta'_{3,c+4}$ (Proposition~\ref{newmainprop}).
  Similarly we get the conclusion for $\beta'_{1,i} \cdot \beta'_{2,i}$ 
  by Corollary~\ref{mainrecor}.
  \end{proof}

  So generic curves may have ghost terms in degree $c+3$ (only). To find an
  example, recall that if we link $C$ to a curve $D$ using a general CI of
  type $(f,g)$ such that $H^1(\sI_C(v)) = 0 \ {\rm for} \ v= f , g , f - 4 \
  {\rm and} \ g - 4,$ then $C$ is generic if and only if
  $D$ is generic (\cite{K3}, Prop.\;3.8).
  \begin{example} \label{ghex}
Using this we 
  take two general skew lines as in Example~\ref{ex1} and we link twice, first
  via a CI of type $(5,2)$, then via a CI of type $(5,4)$. This gives us a
  curve $X$, generic in $\HH(12,18)$, with minimal resolution and a ghost term
  $R(-5)$ in degree $c+3$:
  \begin{equation*} 0 \rightarrow R(-6) \rightarrow R(-7) \oplus R(-5)^{4}
    \rightarrow R(-5) \oplus R(-4)^4 \rightarrow I(X) \rightarrow 0 \, .
    \end{equation*} 
  \end{example}
  Since our concern is about irreducible components of $\HH(d,g)$ containing
  $(C)$, it is only the graded Betti numbers in the 5-tuple and e.g.\;ghost
  terms there that play a role, as we now shall see.

  \begin{definition} Let $C$ be a diameter-1 curve and denote its 5-tuple by
    $\underline \beta(C)_5$. We say a 5-tuple $\underline \beta'_5$
    specializes to $\underline \beta(C)_5$, and we write $\underline \beta'_5
    \leadsto \underline \beta(C)_5$ if we obtain $\underline \beta'_5$ from
    $\underline \beta(C)_5$ by repeatedly using some of the operations {\rm
      (Pi)} for $i =1,2$ and {\rm (Qj)} for $j=c, c+4$ in some order. A
    5-tuple $\underline \beta'_5$ is called minimal if it has the property
    that it does not allow further reductions by using the mentioned
    operations, i.e. $\underline \beta'_5$ is given as in
    Corollary~\ref{cormaxcomp}.
\end{definition}


\begin{theorem} \label{Vbetta} Let $C \subseteq \proj{3}$ be a Buchsbaum curve
  of diameter one. Then there is a one-to-one correspondence between the set
  of minimal 5-tuples that specialize to $ \underline \beta(C)_5$ via the
  operations ${\rm (PQ_J)}$ for $J=\{c, c+4\}$, and the set of irreducible
  (non-embedded) components of $\HH(d,g)$ containing $(C)$, i.e.
  \[ \{ {\rm minimal} \ \underline \beta'_5 \arrowvert \, \underline \beta'_5
  \leadsto \underline \beta(C)_5 \} \stackrel{1-1}{\longleftrightarrow}
  \{ {\rm irreducible \ components \ } V \subset \HH(d,g) \arrowvert \, V \ni
  (C) \} \, .
\]
Here $V$ maps to the 5-tuple of its generic curve and all components
$V$ are generically smooth.
\end{theorem}

\begin{proof} Let $\underline \beta'_5$ be a minimal 5-tuple that specializes
  to $ \underline \beta(C)_5$. We want to define the corresponding irreducible
  component $V(\underline \beta'_5)$ whose generic curve has $\underline
  \beta'_5$ as its 5-tuple. Since the operations (Pi) for $i =1,2$ and (Qj)
  for $j=c, c+4$ on 5-tuples correspond to the existence of generizations,
  there is a generization $\tilde C$ of $C$ in $\HH(d,g)$ such that
  $\underline \beta(\tilde C)_5= \underline \beta'_5$. Then $\tilde C$ is
  unobstructed by Corollary~\ref{introth3}. Let $V(\tilde C)$ be the unique
  irreducible component of $\HH(d,g)$ containing $(\tilde C)$ and let $C'$ be
  the generic curve of $V(\tilde C)$. Then $\underline \beta(C')_5$ is minimal
  by Corollary~\ref{cormaxcomp} and we have $\underline \beta(C')_5 \le
  \underline{\beta}(\tilde C)_5$ by the semi-continuity of 5-tuples, whence
  equality by the minimality of $\underline{\beta}(\tilde C)_5$. Put
  $V(\underline \beta'_5):=V(\tilde C)$.

  To see that the application $\underline \beta'_5 \leadsto V(\underline
  \beta'_5)$ is injective, we suppose $V(\underline{\beta'_1}_5)=
  V(\underline{\beta'_2}_5)$. Then we can assume that their generic curves
  $C_1'$ and $C_2'$ coincide and we conclude the injectivity
  by $$\underline{\beta'_1}_5 = \underline{\beta}(C_1')_5 =
  \underline{\beta}(C_2')_5 = \underline{\beta'_2}_5 \, .$$

  The surjectivity of the application follows from Corollary~\ref{cormaxcomp}
  which implies that a generic curve $C'$ is obtained by taking generizations
  in $\HH(d,g)$ (starting with $C$) using (Pi) and (Qj) in some order. The
  corresponding operations (Pi) and (Qj) on the 5-tuples imply that
  $\underline \beta(C')_5$, which is minimal, specializes to $\underline
  \beta(C)_5$ using only (Pi) and (Qj) for $j=c, c+4$.
 \end{proof}

 \begin{remark} \label{finrem} Theorem~\ref{Vbetta} significantly generalizes
   Prop.\;4.6 of \cite{krao}. It also allows us to interpret geometrically the
   obstructedness result of \cite{krao}, Thm.\,1.3, see
   Corollary~\ref{introth3}. Indeed given $( \beta_{1,c+4},
   \beta_{1,c},\beta_{2,c+4},\beta_{2,c},\beta_{3,c+4})$ with $\beta_{3,c+4}
   \ne 0$, then the obstructedness condition {$$ \beta_{1,c} \cdot
     \beta_{2,c+4} \neq 0 \ \ \ {\rm {\bf or}} \ \ \ \beta_{1,c+4} \cdot
     \beta_{2,c+4} \neq 0 \ \ \ {\rm {\bf or}} \ \ \ \beta_{1,c} \cdot
     \beta_{2,c}\neq 0 $$}is equivalent to the following statement: there
   exist generizations given by {\rm (P1)} and {\rm (P2)}, {\bf or} {\rm (P1)}
   and {\rm (Q(c$\,+\,4$))}, {\bf or} {\rm (P2)} and {\rm (Qc)} respectively,
   where each of the three ``and''-expressions correspond to two different
   (``directions for the'') generizations, removing at least one ghost term in
   a minimal resolution of $I(C)$. Moreover each of the three expressions may
   correspond to two different irreducible components of $\HH(d,g)$, but not
   necessarily, as we may see from:
\end{remark}

\begin{example} \label{ex33} {\rm (a)} The obstructed curve $C$ of
  Example~\ref{ex3} (a) has 5-tuple $(0,1,1,0,2)$. It admits two generizations
  to two curves with 5-tuples $(0,1,0,0,1) \ {\rm and } \ (0,0,1,0,1).$ These
  5-tuples are not minimal. Indeed both curves admit generizations to curves
  with the same 5-tuple $(0,0,0,0,0)$. By Theorem~\ref{Vbetta}\, $C$ belongs
  to a unique
  irreducible components of $\HH(32,109)_S$!  \\[2mm]
  {\rm (b)} 
  The 5-tuple of the obstructed curve $C$ in Example~\ref{ex3} (b) is
  $(1,0,1,0,1)$, i.e. the curve $C$ admits two generizations to two curves
  with minimal 5-tuples $(1,0,0,0,0) \ {\rm and } \ (0,0,0,0,1)$, where one of
  the generizations is ACM and the other is Buchsbaum of diameter one. By
  Theorem~\ref{Vbetta}\, $C$ belongs to exactly two irreducible components of
  $\HH(33,117)_S$, cf. \cite{BKM}. Note that both generizations correspond to
  the removal of ghost terms, cf. \cite{W2}, Ex.\;4.2. Hence we can not
  separate the two
  components by the usual semi-continuity of $h^i(\sI_{C}(v))$!  \\[2mm]
  {\rm (c)} The 5-tuple of the curve $X$ of Example~\ref{ex1} is
  $(0,1,1,0,1)$, having two generizations with 5-tuples $(0,1,0,0,0) \ {\rm
    and } \ (0,0,1,0,0).$ These 5-tuples are minimal and the corresponding
  curves are ACM. By Theorem~\ref{Vbetta} there are precisely two irreducible
  components $V_1$, $V_2$ of $\HH(18,39)_S$ such that $(X) \in V_1 \cup V_2$,
  cf. Example~\ref{ex2}. Note that we in this case may separate the two
  components by the semi-continuity of $h^i(\sI_{Z}(v))$ because 
  $(h^0(\sI_{Z}(4)), h^1(\sI_{Z}(4)), h^1(\sO_{Z}(4)))$ is equal to $(1,1,1)$
  for $Z=X$, while it is $(1,0,0)$ and $(0,0,1)$ for the two generizations.
\end{example}

Our next proposition and remark, which was communicated to us by Johannes
Kleppe together with a full proof and Example~\ref{johan}, determine
explicitly how many irreducible components of $\HH(d,g)$ that we have in the
correspondence given in Theorem~\ref{Vbetta}. Below $(a_1,a_2,b_1,b_2,r) =
(\beta_{1,c+4},\beta_{1,c}, \beta_{2,c+4}, \beta_{2,c}, \beta_{3,c+4})$ and we
let $\binom mn = 0$ if $m<n$.

\begin{proposition} \label{numberBCM} Let $(a_1,a_2,b_1,b_2,r)$ be the 5-tuple
  of a Buchsbaum curve of diameter one, and let $\hat a_2 = \max
  \{0,a_2-b_2\}$ and $\hat b_1 = \max \{0,b_1-a_1\}$. The number of minimal
  5-tuples that specialize to $(a_1,a_2,b_1,b_2,r)$ is
  \begin{equation*}
    N_B + N_{CM}.
  \end{equation*} \\[-7mm]
  Here,
  \begin{equation}
    \label{eq:NB}
    N_B = \binom {r-\hat b_1-\hat a_2+1} 2 - \binom {r-b_1-\hat a_2} 2 -
    \binom {r-\hat b_1-a_2} 2 + \binom {r-b_1-a_2-1} 2
  \end{equation}
  is the number of minimal 5-tuples that correspond to generic diameter-1
  curves, and
  \begin{equation}
    \label{eq:NCM}
    N_{CM} =
    \begin{cases}
      \min\{b_1,a_2,r\}+1, & \text{if } r \le \max \{b_1,a_2\} \\
      b_1+a_2-r+1, & \text{if }  \max \{b_1,a_2\} \le r \le b_1+a_2 \\
      0, & \text{if } r > b_1+a_2 \\
    \end{cases}
  \end{equation}
  is the number of minimal 5-tuples that correspond to generic $ACM$ curves.
\end{proposition}

\begin{proof}
  The four basic reductions of a 5-tuple $(a_1,a_2,b_1,b_2,r)$ are given by
  the vectors $\underline \alpha_1 = (1,0,1,0,0)$, $\underline \alpha_2 =
  (0,0,1,0,1)$, $\underline \alpha_3 = 
  (0,1,0,0,1)$ and $\underline \alpha_4 = (0,1,0,1,0)$. Any reduction of the
  5-tuple can 
  be written as
  \begin{equation*}
    (a_1,a_2,b_1,b_2,r) - \sum_{i=1}^4 k_i \underline \alpha_i = (a_1-k_1,
    a_2-k_3-k_4, 
    b_1-k_1-k_2, b_2-k_4, r-k_2-k_3)
  \end{equation*}
  with each $k_i \ge 0$. These numbers cannot be negative, giving us the
  following five inequalities:
  \begin{equation*}
    k_1 \le a_1 \qquad k_1+k_2 \le b_1 \qquad k_2+k_3 \le r \qquad k_3+k_4 \le
    a_2 \qquad k_4 \le b_2 \,.
  \end{equation*}
  Clearly, we have arrived at a minimal 5-tuple if and only if no $k_i$ can be
  increased, implying that among each pair of neighbouring inequalities in the
  above, one must be an equality. To count the number of minimal 5-tuples, we
  will divide into two cases, depending on whether $r$ is reduced to zero or
  not.


  \emph{Case} 1. If $r$ is non-zero in the minimal 5-tuple, then the
  reductions of $b_1$ and $a_2$ must both be zero. Hence the minimal 5-tuple
  is of the form $(*,0,0,*,+)$, giving the following:
  \begin{equation*}
    k_1 \le a_1 \qquad k_1+k_2 = b_1 \qquad k_2+k_3 < r \qquad k_3+k_4 = a_2
    \qquad k_4 \le b_2 \,.
  \end{equation*}
  This requires that $k_1 \le \min\{a_1,b_1\}$, and therefore $k_2 \ge b_1 -
  \min\{a_1,b_1\} = \max\{0,b_1-a_1\} = \hat b_1$. Conversely, $\hat b_1 \le k_2
  \le b_1$ implies $0 \le k_1 \le \min\{a_1,b_1\}$. Hence the minimal 5-tuples
  in \emph{Case} 1 are in one-to-one correspondence with all pairs $(k_2,
  k_3)$ within the square $\hat b_1 \le k_2 \le b_1$ and $\hat a_2 \le k_3 \le
  a_2$ that satisfy $k_2 + k_3 < r$. The number of such pairs can be expressed
  using triangular numbers as
  \begin{equation*}
    N_B = \binom {r-\hat b_1-\hat a_2+1} 2 - \binom {r-b_1-\hat a_2} 2 -
    \binom {r-\hat b_1-a_2} 2 + \binom {r-b_1-a_2-1} 2 \,.
  \end{equation*}
  Note that $N_B \le (\min \{a_1,b_1\} +1)(\min \{a_2,b_2\} +1)$, with
  equality if and only if $r > b_1+a_2$.

  \emph{Case} 2. If $r$ is reduced to zero, we get a 5-tuple of the
  form $(*,*,*,*,0)$. This form is a specialization of a unique minimal
  5-tuple, found by reducing the pairs $(a_1,b_1)$ and $(a_2,b_2)$, i.e.
  increasing $k_1$ and $k_4$, until one of the integers in each pair reach
  zero. Therefore, we only have to count in how many ways $r$ can be reduced
  to zero, and the constraints for these minimal 5-tuples are as follows:
  \begin{equation*}
    k_1 = \min \{a_1, b_1-k_2\} \qquad k_2 \le b_1 \qquad k_2+k_3 = r \qquad
    k_3 \le 
    a_2 \qquad k_4 = \min \{b_2, a_2-k_3\} \,.
  \end{equation*}
  In other words, the minimal 5-tuples in \emph{Case} 2 correspond to those
  pairs $(k_2, k_3)$ on the line $k_2+k_3 = r$ that satisfy $k_2 \le b_1$ and
  $k_3 \le a_2$, implying formula \eqref{eq:NCM}.
\end{proof}

\begin{example} \label{johan}
  Let us count the number of minimal 5-tuples that specialize to $(3,7,5,5,6)$
  (disregarding if this is a 5-tuple of a diameter-1 curve that exists). In
  this case $\hat b_1 = b_1-a_1 = 2$ and $\hat a_2 = a_2-b_2 = 2$. The minimal
  5-tuples are easily visualized in the $k_2k_3$-plane:

  \newlength{\scalefactor}%
  \setlength{\scalefactor}{1.6pt}%
  \newlength{\miniwidth}%
  \setlength{\miniwidth}{\linewidth}%
  \addtolength{\miniwidth}{-85\scalefactor}%
  \noindent%
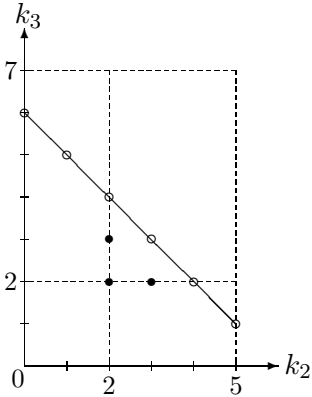
  \begin{minipage}[c]{85\scalefactor}
    \setlength{\unitlength}{\scalefactor}
    \begin{picture}(85,100)(-15,-10)
      \small
      \linethickness{.1pt}
      \put (0,0) {\vector (1,0) {60}}%
      \put (0,0) {\vector (0,1) {80}}%
      \put (61.5,0) {\makebox(0,0)[l]{$k_2$}}%
      \put (1,83) {\makebox(0,0)[c]{$k_3$}}%
      \multiput (10,-1) (10,0) {5} {\line (0,1) {2}}%
      \multiput (-1,10) (0,10) {7} {\line (1,0) {2}}%
      \footnotesize
      \put (0,-1) {\makebox(0,0)[tr]{$0$}}%
      \put (20,-2.5) {\makebox(0,0)[t]{$2$}}%
      \put (50,-2.5) {\makebox(0,0)[t]{$5$}}%
      \put (-1.7,20) {\makebox(0,0)[r]{$2$}}%
      \put (-1.7,70) {\makebox(0,0)[r]{$7$}}%
      \multiput (20,1) (0,2) {35} {\line (0,1) {1}}%
      \multiput (50,1) (0,2) {35} {\line (0,1) {1}}%
      \multiput (1,20) (2,0) {25} {\line (1,0) {1}}%
      \multiput (1,70) (2,0) {25} {\line (1,0) {1}}%
      \linethickness{.5pt}
      \put (0,60) {\line (1,-1) {50}}%
      \multiput (0,60) (10,-10) {6} {\circle 2}%
      \put (20,20) {\circle* 2}%
      \put (20,30) {\circle* 2}%
      \put (30,20) {\circle* 2}%
    \end{picture}%
  \end{minipage}%
  \hspace{15pt}\addtolength{\miniwidth}{-15pt}%
  \begin{minipage}[c]{\miniwidth}
    The minimal 5-tuples counted by $N_B$
    are determined by the points inside the rectangle $2 \le k_2 \le 5$ and $2
    \le k_3 \le 7$ below the line $k_2+k_3=6$. These are marked as filled
    dots. We see that $N_B = 3$.

    \medskip

    The minimal 5-tuples counted by $N_{CM}$
    are given by the points on the line $k_2+k_3=6$ inside the larger
    rectangle $0 \le k_2 \le 5$ and $0 \le k_3 \le 7$. These are marked as
    open dots. We easily count that $N_{CM} = 6$.

    \medskip

    In total we have $N_B + N_{CM} = 9$ different minimal 5-tuples.

  \end{minipage}
\end{example}

\begin{remark}
  In some cases there is only one minimal 5-tuple that specializes to a given
  5-tuple $(a_1,a_2,b_1,b_2,r)$. This happens if and only if the original
  5-tuple has the following property: if $(x,y,z)$ is any of the triplets
  $(a_1,b_1,r)$, $(b_1,r,a_2)$ or $(r,a_2,b_2)$, then either $xyz=0$ or $y \ge
  x+z$. Indeed, each of these triplets have two possible basic reductions,
  given by the vectors $(1,1,0)$ and $(0,1,1)$. If there is a unique minimal
  5-tuple, then also these triplets must have a unique reduced version, and
  this is equivalent to the stated property. Note that this implies that the
  sequence $(a_1,b_1,r,a_2,b_2)$ cannot have 4 neighbouring positive integers.

  In addition to four obvious cases (namely $r=0$, $a_1=a_2=0$, $a_2=b_1=0$
  and $b_1=b_2=0$), this gives us the following three cases: $a_2=0$ and $b_1
  \ge r+a_1$, $b_1=0$ and $a_2 \ge r+b_2$, or $a_1=b_2=0$ and $r \ge b_1+a_2$.
  Example~\ref{ex33} {\rm (a)} belongs to the last case. An example of each of
  the other two main cases is given below.
\end{remark}
 
\begin{example} \label{Wa} {\rm (a)} There is an obstructed curve $C$ in
  $\HH(42,177)_S$
 with minimal resolution
\begin{equation*}
  0 \rightarrow R(-10) \rightarrow R(-11)^{2} \oplus  R(-10)^2
  \oplus  R(-9)^{4} \rightarrow 
  R(-10) \oplus  R(-9)^2 \oplus
  R(-8)^{5} \rightarrow I \rightarrow 0 \,  
  \end{equation*}
  (\cite{W2}, Ex.\;4.2). Since the 5-tuple of $C$ is $(1,0,2,0,1)$, it admits
  two generizations to curves with 5-tuples; $(1,0,1,0,0) \ {\rm and } \
  (0,0,1,0,1)$. 
  These 5-tuples are not minimal. Indeed both curves admit generizations to
  curves with the same 5-tuple $(0,0,0,0,0)$. By Theorem~\ref{Vbetta}, \ $C$
  belongs to a unique component of $\HH(42,177)$. Moreover since all
  generizations above correspond to the removal of ghost terms, they preserve
  postulation. It follows that $(C)$, which is a singular point of $\HH_{
    \gamma}=\HH(42,177)_{ \gamma}$, belongs to a unique component of $\HH_{
    \gamma}$ (or one may use that $ {_0\!\Hom_R}(I(C),M(C)) = 0$ implies
  $\HH_{ \gamma} \cong \HH(42,177)$ at $(C)$, cf.\;Theorem~\ref{gen}, to see
  it). 

  {\rm (b)} If we link the curve of  {\rm (a)} via a CI of type $(8,8)$ we get
  an obstructed curve $D$ 
  with 5-tuple $(0,2,0,1,1)$. The curve $D$ admits two generizations to two
  curves with 5-tuples $(0,1,0,0,1) \ {\rm and } \ (0,1,0,1,0)$,
  and two further generizations to
  curves with the same 5-tuple $(0,0,0,0,0)$. By Theorem~\ref{Vbetta}, \ $D$
  belongs to a unique irreducible components of $\HH(22,57)$. 
 \end{example}  
 \vspace*{-0.17in}

\section{The Hilbert scheme of  curves of diameter at most one}
 \vspace*{-0.07in}
In this section we study 
the open subscheme, $ \HH(d,g;c)$, of $ \HH(d,g)$ whose $k$-points are given by
\begin{equation*}
  \{(C) \in  \HH(d,g) \arrowvert  \  H^1(\sI_{C}(v)) = 0 \ { \rm for \ every \
  } v  \ne c  \} \, ,
  \end{equation*}
  $c$ an integer.
  Our main concern is to determine its singular locus. To do so,
  Theorem~\ref{newmain}, which describe ``all'' generizations of curves in $
  \HH(d,g;c)$, together with the characterization of obstructed curves in
  Corollary~\ref{introth3}, will be the main
  ingredient. 
  Note that Theorem~\ref{Vbetta}, whose proof strongly needed
  Theorem~\ref{newmain}, 
  directly transfers to a theorem for $ \HH(d,g;c)$ with similar statements
  because all components of Theorem~\ref{Vbetta} properly intersect $
  \HH(d,g;c(C))$.
%

  In the following let $C$, $(C) \in \HH(d,g;c)$, be a generic curve of a
  Betti stratum $\HH(\underline \beta)$, and let $\underline \beta_5$ be the
  5-tuple of $C$. We write $\HH(\underline \beta)$ as $\HH(\underline
  \beta_5)$ if the graded Betti numbers that do not belong to $\underline
  \beta_5$ are chosen as small as possible (cf. Corollary~\ref{mainrecor}),
  i.e. so that they satisfy
 \begin{equation} \label{pluss}
 \ \beta_{1,c+3} \cdot (\beta_{2,c+3}-4 \beta_{3,c+4})=0  ,  \
 \beta_{1,i} \cdot \beta_{2,i} = 0 \ \ {\rm for \ } i \notin \{c,c+3,c+4\} \, .
\end{equation}
Note that if ${\overline \HH}(\underline \beta)$, ${\overline \HH}(-)$ the
  closure of $ { \HH}(-)$ in $ \HH(d,g)$, is an irreducible component of $
  \HH(d,g)$, then $\underline \beta$ satisfies \eqref{pluss} by
  Corollary~\ref{cormaxcomp}. Suppose $\HH(\underline \beta) = \HH(\underline
  \beta_5)$, i.e. that $C$ satisfies \eqref{pluss}, and let $V(\underline
  \beta_5)_{B}:= {\overline \HH}(\underline \beta_5) \cap \HH(d,g;c)$. If
  $(C') \in V(\underline \beta_5)_{B}$ then $C$ is a generization of $C'$ in
  $\HH(d,g)$ generated by (PQ) by Theorem~\ref{newmain}, see also \cite{H1},
  Ch.\;II, Ex.\;3.17. 
  Now we denote by \[ {\underline p}_1:=(0,0,1,0,1), \ {\underline
    p_2}:=(0,1,0,0,1), \ {\underline q_c}:=(0,1,0,1,0), \ {\underline
    q_{c+4}}:=(1,0,1,0,0) \] the vectors that correspond to the operations
  (P1), (P2) and (Qj) for $j=c, c+4$ respectively. 
We define
 \begin{equation}
    V(\underline \beta_5 + \underline q_J)_{B} :=
    \begin{cases}
      V(\underline \beta_5 + \underline q_c)_{B} \cup V(\underline
      \beta_5 + \underline q_{c+4})_{B}\, , & \text{if \ diam}\; M(C)=1 \\
      \ \emptyset & \text{if} \ C \ {\rm is  \ ACM} \,. \\
      \end{cases}
  \end{equation}
  Below $+$, resp.\;$*$ \, in an entry of a 5-tuple means a positive,
  resp.\;non-negative integer. Moreover if $V(\underline \beta_5)_B$ is an
  irreducible component of \ $\HH(d,g;c)$, then we denote by $Sing \
  V(\underline \beta_5)_B$ the part of the singular locus of $\HH(d,g;c)$ that
  are contained in $ V(\underline \beta_5)_B$. We get
  \begin{theorem} \label{singloc} With the above notations, suppose
    $V(\underline \beta_5)_B$ is an irreducible component of \
    $\HH(d,g;c)$. 
    Then $\underline \beta_5$ is given as in {\rm (i)-(v)}, and
    \\[2mm]
    {\rm (i)} \hspace{0.2cm} if $\underline \beta_5$ is equal to {\rm
      ($+,0,0,+,*$) \ or \
      ($0,+,+,0,0$), \hspace{0.1cm}} then \\[-2mm]
    \[ Sing \ V(\underline \beta_5)_B = V(\underline \beta_5 + \underline
    p_1)_{B} \cup V(\underline \beta_5 + \underline p_{2})_{B} \cup
    V(\underline \beta_5 + \underline q_J)_{B},
    \]
    \\[-2mm]
    {\rm (ii)} \hspace{0.2cm} if {\rm $\underline \beta_5 = \ $($0,0,0,+,*$) \
      or \ ($0,0,+,*,0$)}, \ then \ $ Sing \ V(\underline \beta_5)_B =
    V(\underline \beta_5 + \underline p_{2})_{B} \cup V(\underline \beta_5 +
    \underline q_J)_{B}$,
    \\[2mm]
    {\rm (iii)} \hspace{0.05cm} if {\rm $\underline \beta_5 = \ $($+,0,0,0,*$)
      \ or \ ($*,+,0,0,0$),} \ then \ $ Sing \ V(\underline \beta_5)_B =
    V(\underline \beta_5 + \underline p_1)_{B} \cup V(\underline \beta_5 +
    \underline q_J)_{B}$,
    \\[2mm]
    {\rm (iv)} \hspace{0.1cm} if {\rm $\underline \beta_5 = \ $($0,0,0,0,+$)},
    \ then $ \ Sing \ V(\underline \beta_5)_B = V(\underline \beta_5 +
    \underline p_1+\underline p_{2})_{B} \cup V(\underline \beta_5 +
    \underline q_J)_{B}$,
    \\[2mm]
    {\rm (v)} \hspace{0.2cm} if {\rm $\underline \beta_5 = \ $($0,0,0,0,0$)},
    \ \ then \hspace{7mm} $ \ \ \ Sing \ V(\underline \beta_5)_B \ \ = \ \
    V(\underline
    \beta_5 + \underline p_1+\underline p_{2})_{B} \ \cup \\[3mm]
    {} \hspace{12mm}V(\underline \beta_5 + \underline p_1 + \underline
    q_c)_{B} \cup V(\underline \beta_5 + \underline p_1 + \underline
    q_{c+4})_{B} \cup V(\underline \beta_5 + \underline p_2 + \underline
    q_c)_{B} \cup V(\underline \beta_5 + \underline p_2 + \underline
    q_{c+4})_{B}\, .$
    \\[1mm]
     \end{theorem}
     \begin{proof} 
       It is easily checked that the minimal 5-tuples are of the form
       (i)-(v). Now let $C$ be a generic curve of $V(\underline
       \beta_5)_{B}$.

       (i) A generic curve $\tilde C$ of a non-empty $V(\underline \beta_5 +
       \underline p_1)_{B}$ has 5-tuple without consecutive 0's in its first 4
       entries, whence $\tilde C$ is obstructed by
       Remark~\ref{rem5tupleobstr}. The same argument, using
       Remark~\ref{rem5tupleobstr}, holds for $V(\underline \beta_5 +
       \underline p_2)_{B}$. If $C$ is not ACM, the argument also holds for
       the generic curve $\tilde C$ of $V(\underline \beta_5 + \underline
       q_i)_{B}$, $i=c$ and $c+4$. Since $C$ 
       is a generization of  $\tilde C$,
       it follows that $(\tilde C)$ belongs to the closure of $\HH(\underline
       \beta_5)$ in $\HH(d,g)$, 
       i.e. that $(\tilde C) \in V(\underline \beta_5)_B$ and we get
       \[ Sing \ V(\underline \beta_5)_B \supseteq V(\underline \beta_5 +
       \underline p_1)_{B} \cup V(\underline \beta_5 + \underline
       p_{2})_{B} \cup V(\underline \beta_5 + \underline q_J)_{B}.
    \]
     (also in the case some of the Betti strata to the right are empty).

    Conversely suppose a curve $C'$ of $ V(\underline \beta_5)_B$ is not in
    the union of the $V$-sets above. If the generic curve $ C$ of
    $V(\underline \beta_5)_{B}$ is not ACM, then $C$ is by
    Theorem~\ref{newmain} a generization of $C'$ in $\HH(d,g)$ generated by
    (PQ) without using (P1), (P2), nor (Qi) for $i=c$ and $c+4$. This follows
    from the fact that we can change the order in which we use (Pj) and (Qi).
    Indeed if e.g. (P2) is used, then $\underline \beta_5 + \underline p_2$
    must specialize to the 5-tuple of $C'$ which implies that $(C')$ belongs
    to the closure of $\HH(\underline \beta_5 + \underline p_2)$
    and we get a
    contradiction. 
    Thus $ C$ is a trivial generization of $C'$, which implies that $C'$ has
    exactly the same 5-tuple as $C$. It follows that $C'$ is unobstructed.

    If $C$ is ACM, then $C$ is a generization of $C'$ in
    $\HH(d,g)$ generated by (PQ) without using (P1) nor (P2), i.e. only
    generizations given by (Qi) are used. Then $C'$ is ACM and hence
    unobstructed. This proves (i).

    The other cases (ii)-(v) are proven similarly, and we get the theorem.
\end{proof}

Finally we remark that we can find the dimension of the singularities given in
Theorem~\ref{singloc} in some cases. Indeed let $ {\HH}(\underline \beta_5)
\subseteq {\HH}_{\gamma,\rho}$ be a Betti stratum with generic curve $C$, $(C)
\in \HH(d,g;c)$, and let $C'$ be a generic curve of ${\HH}_{\gamma,\rho}$
satisfying \eqref{pluss} by Theorem~\ref{mainres}. Then $C'$ is a generization
of $C$ in ${\HH}(d,g)$ without using (P1) and (P2). Indeed (P1) and (P2)
change $\rho$. It follows that $C'$ is a generization of $C$ generated by
${\rm (Q_J)}$, $J= \{c,c+4\}$. Suppose $\underline \beta_5=\underline
\beta_5(C)$ is of the form
 \begin{equation} \label{formbeta}
\underline \beta_5 = (0,
\beta_{1,c},\beta_{2,c+4},0,\beta_{3,c+4}).
\end{equation}
Then neither (Qc) nor (Q(c$\,+\,4$)) are used, i.e. $C'$ is a trivial
generization of $C$ and $ (C') \in {\HH}(\underline \beta_5)$. It follows that
$ V(\underline \beta_5)_B = \overline{\HH}_{\gamma,\rho} \cap \HH(d,g;c)$. 
Since $\dim {\HH}_{\gamma,\rho}$ is known (\cite{krao}, Rem.\;2.3, first
proved in \cite{MDP1}, Thm.\;3.8, p.\;171), we can compute the dimension of
the singularities $ V(\underline \beta_5 + a \underline p_1 +b \underline
p_2)_{B}$ for $a,b \in \{0,1\}$, of Theorem~\ref{singloc} 
because their generic curves satisfy
\eqref{formbeta}:

\begin{example} \label{exsingloc} {\rm (a)} The singularity ``$(0,1,1,0,2)$''
  of Example~\ref{ex33} (a) belongs to a unique irreducible component of
  $\HH(32,109)_S$ with 5-tuple $(0,0,0,0,0)$.
  The  codimension of the singularity, i.e.   $\dim V(0,0,0,0,0)_{B}- \dim V(0,1,1,0,2)_{B}$, is\, $3$.  \\[2mm]
  {\rm (b)} By \cite{krao}, Ex.\, 3.12, there exists a singularity
  ``$(0,1,1,0,r)$'' belonging to a unique irreducible component of
  $\HH(d,g)_S$ for any $r \ge 2$, and the
  codimension of the singularity is $2r-1$.  \\[2mm]
  {\rm (c)} 
  The singularity of Example~\ref{ex33} (c) sits in the intersection of two
  irreducible components of $\HH(18,39)_S$,
  and 
  the codimension of the singularity in each of its components is $1$ (cf.
  \cite{Se} and \cite{EF}).
\end{example}

\bigskip \bigskip


\end{document}